\documentclass{amsart}

\usepackage{amssymb}

\usepackage{mathrsfs}
\usepackage{graphicx}
\usepackage{comment}
\usepackage[francais,english]{babel}
\usepackage[T1]{fontenc}
\numberwithin{equation}{section}
\def\P{\mathcal P}
\newcommand{\comments}[1]{}

\newcommand{\be}{\begin{equation}}
\newcommand{\ee}{\end{equation}}
\newcommand{\ba}{\begin{align}}
\newcommand{\ea}{\end{align}}

\newtheorem{theorem}{Theorem}[section]
\newtheorem{corollary}[theorem]{Corollary}
\newtheorem{lemma}{Lemma}[section]

\DeclareMathOperator{\even}{even}
\DeclareMathOperator{\odd}{odd}
{\begin{list}{}{%
\settowidth{\labelwidth}{\textsf{{\it #1.}}}%
\setlength{\labelsep}{4mm}%
\setlength{\leftmargin}{\labelwidth}%
\addtolength{\leftmargin}{\labelsep}%
}}%
{\end{list}}

\def\Z{\mathbb Z}
\def\S{\mathcal S}
\def\om{\omega}
\def\midd{\, : \,}
{\begin{list}{}{%
\settowidth{\labelwidth}{\textsf{{\it #1.}}}%
\setlength{\labelsep}{2mm}%
\setlength{\leftmargin}{\labelwidth}%
\addtolength{\leftmargin}{\labelsep}%
\addtolength{\leftmargin}{4mm}%
\setlength{\itemsep}{6pt}%
\setlength{\listparindent}{0pt}%
\setlength{\topsep}{3pt}%
}}%

\title[Integer group determinants]{ Integer group determinants for small groups}

\author[C. Pinner]{Christopher Pinner}
\address{ Department of Mathematics\\
         Kansas State University\\
         Manhattan, KS 66506, USA}
\email{pinner@math.ksu.edu}
\thanks{The second  author thanks the University of Edinburgh for the invitation 
 to visit,  and the Edinburgh Mathematical Society for its financial  support.}

\author[C. Smyth]{Christopher Smyth}
\address{School of Mathematics and Maxwell Institute for Mathematical Sciences\\
University of Edinburgh\\
Edinburgh EH9 3FD\\
Scotland, UK}
\email{c.smyth@ed.ac.uk}

\date{\today}
\begin{document}

\begin{abstract}
For every group of order at most 14 we determine  the values taken  by its
group determinant when its variables are integers.

\end{abstract}

\maketitle

\section{Introduction}\label{secIntroduction}
For a finite group $G=\{g_1,\ldots ,g_n\}$ of order $n,$ we assign a  variable 
$x_g$ for each element $g\in G$ and define its {\it group determinant}
$\mathscr{D}_G(x_{g_1},\ldots ,x_{g_n})$ to be the determinant of the $n\times 
n$ matrix whose $(i,j)$th entry is $x_{g_i g_j^{-1}}$.
In the case of the cyclic group of order $n$, the group determinant becomes an 
$n\times n$ {\it circulant determinant}, where  
each row is obtained from the previous one by a cyclic shift one step to the 
right. At the meeting of the American Mathematical Society in Hayward, 
California, in April 1977,  Olga Taussky-Todd asked which integers could be obtained 
as an $n\times n$ circulant 
determinant when the entries are all integers.  Of course we can ask for a 
complete description of the group determinants over the integers for any group $G$, 
not just for the cyclic groups. Thus our problem is to determine the set
\[
  \S(G)=\{  \mathscr{D}_G(x_{g_1},\ldots ,x_{g_n})\midd x_{g_1},\ldots ,x_{g_n}\in\Z \}.                                                                          
   \]
For the additive cyclic group $\mathbb Z_n$ of order $n$, 
  Laquer \cite{Laquer} and Newman \cite{Newman1,Newman2} gave 
divisibility conditions on the integers that can be group determinants, as well as sets of 
achievable values; for example any integer coprime to $n$ or a multiple of 
$n^2$ 
will be a group determinant, if $m$ is a determinant then so is $-m$, and if 
$p\mid m$ and $p^{\alpha}\parallel n$ then $p^{\alpha+1}\mid m$. Conditions like 
these enabled them to obtain a complete description  of the values for certain 
cyclic groups. 
For example   Laquer \cite{Laquer} and Newman 
\cite{Newman1} showed that for a prime $p$
\be \label{Zp}  \S(\Z_p)=\{p^a m_p\midd a=0, a\geq 2\}, \ee
while for  $p$ an odd prime, Laquer \cite{Laquer} showed that 
\be \label{Z2p}  \S(\Z_{2p})=\{2^a p^b m_{2p}\midd  a=0, a\geq 2, \;\; b=0, 
 b\geq 2\}. \ee
Newman \cite{Newman2} determined $\S(\Z_9)$ as
\be \label{Z9}   \S(\Z_9)=\{3^am_3\midd  a=0, a\geq 3\},     \ee
with upper and lower set  inclusions for general $\mathbb Z_{p^2}$.  In the above, and henceforth, $m_t$ denotes an arbitrary integer coprime to $t$.

For a polynomial $F(x_1,\ldots ,x_r)$ in $\mathbb Z [x_1,\ldots ,x_r]$, the 
traditional logarithmic Mahler measure $m(F)$ can  be defined by
$$ m(F)=\log M(F) = \int_0^1 \cdots \int_0^1 \log |F(e^{2\pi i x_1},\ldots , 
e^{2\pi i x_r})| dx_1 \cdots dx_r. $$
In 2005 Lind \cite{Lind} viewed the traditional Mahler measure as a measure on 
the circle group $(\mathbb R /\mathbb Z)^r$ and generalised the concept to an 
arbitrary compact abelian group. In particular for a finite group
\be \label{finiteG}  G=\mathbb Z_{n_1}\times \cdots \times \mathbb Z_{n_r} \ee
one can define the logarithmic measure of an $F(x_1,\ldots ,x_r)$ in $\mathbb Z 
[x_1,\ldots ,x_r]$ relative to $G$ to be $$m_G(f)=\frac{1}{|G|} \log |M_G(F)|, 
$$ where
$$ M_G(F)=\prod_{j_1=1}^{n_1} \cdots \prod_{j_r=1}^{n_r}F(\om_{n_1}^{j_1},\ldots 
,\om_{n_r}^{j_r}),\;\; \;\;\; \om_n:=e^{2\pi i/n}. $$

Curiously, the Lind variant of the Mahler measure for $\Z_n$ had essentially appeared in a 1916 paper of Pierce \cite{Pierce}, and also in the
famous paper of Lehmer \cite{Lehmer}, in the form $\Delta_n:=\prod_{i=1}^r(\alpha_i^n-1)=(-1)^{rn} M_{\Z_n}(F)$, for a monic integer one-variable polynomial $F$ with roots $\alpha_1,\ldots ,\alpha_r$.

As observed by Dedekind, the group of characters $\hat{G}$ of a finite abelian 
group $G$ can be used to factor its group determinant as
\be \label{abelianfactor} \mathscr{D}_G(x_{g_1},\ldots ,x_{g_n})=\prod_{\chi \in 
\hat{G}} \left(\chi(g_1)x_{g_1}+\cdots + \chi(g_n)x_{g_n}\right). \ee
On making the characters  explicit, it is readily seen that for a group $G$ of the form 
\eqref{finiteG} we have 
\be \label{connection}  \mathscr{D}_G(a_{g_1},\ldots ,a_{g_n})= M_G(F), \ee
where
\be \label{defF}  F(x_1,\ldots ,x_r) =\sum_{g=(t_1,\ldots t_r)\in G} a_g 
x_1^{t_1}\cdots x_r^{t_r}, \ee
a connection observed by Vipismakul \cite{Cid1} in his thesis. Of course any 
polynomial in $\mathbb Z[x_1,\ldots ,x_r]$ can be reduced to \eqref{defF}  by 
working in the ring $\mathbb Z[x_1,\ldots ,x_r]/\langle x_1^{n_1}-1,\ldots 
,x_r^{n_r}-1\rangle.$

Kaiblinger \cite{Norbert2} used the Lind measure approach to obtain 
\be \label{Z4}   \S(\mathbb Z_4)=\{2^a m_2\midd  a=0, a\geq 4\} \ee
and 
\be \label{Z8}   \S(\mathbb Z_8)=\{2^a m_2 \midd  a=0, a\geq 5\},    \ee
with upper and lower set inclusions for the other $\mathbb Z_{2^k}$. Defining 
$\lambda(G)$ to be the smallest non-trivial determinant value
$$ \lambda(G):= \min \{ |s|\midd s\in\S(G),\; s\neq 0, \pm 1\}, $$
Kaiblinger  \cite{Norbert}  obtained $\lambda (\mathbb Z_n)$ when $420\nmid n$; 
 this was extended to all $n$ with $892371480 \nmid n$
by Pigno and Pinner \cite{Pigno1}. Values of $\lambda(G)$ for non-cyclic abelian $G$  were 
considered in \cite{dilum,pgroups,Cid2, Stian}.

As explored in Boerkoel and Pinner \cite{dihedral}, the connection \eqref{connection} between Lind 
measures and group determinants suggests a way to extend the concept of Lind 
measure to 
non-abelian finite groups, and to measures on (not necessarily commutative) polynomial rings modulo 
appropriate group relations. See Dasbach and Lal\'\i n\cite{Lalin} for another approach. 
As observed 
by Frobenius, see for example \cite{Formanek,Conrad}, 
the counterpart to \eqref{abelianfactor} for a non-abelian group will involve 
non-linear factors and the set of irreducible
representations $\hat{G}$ for $G$. Specifically,
$$ \mathscr{D}_G(x_{g_1},\ldots ,x_{g_n})=\prod_{\rho\in \hat{G}} \det\left( 
\sum_{g\in G} x_g \rho(g)\right)^{\deg(\rho)}. $$
For example, for the dihedral
group $G=D_{2n}$ of order $2n$, one can define the measure $M_G(F)$ of  an $F$ in 
$ \mathbb Z[x,y]/\langle x^n-1,y^2-1, xy-yx^{-1}\rangle,$ reduced to the form
$$ F(x)=f(x)+yg(x),\;\; \;\; f(x)=\sum_{j=0}^{n-1} a_j 
x^j,\;\;\;g(x)=\sum_{j=0}^{n-1} b_j x^j, $$ 
by
\[
  M_G( F)  =\mathscr{D}_G(a_0,\ldots ,a_{n-1},b_0,\ldots 
,b_{n-1}). 
\]
This was shown in \cite[Section 2]{dihedral} to equal 
\be\label{E-D}
M_{\mathbb Z_n}\left(f(x)f(x^{-1})-g(x)g(x^{-1})\right),
\ee
and was used in \cite{dihedral} to determine $\S(D_{2p})$ for 
$p$ an odd prime as 
\be \label{D2p}  \S(D_{2p})=\{2^a p^b m_{2p}\midd  a=0,  a\geq 2,\;\; b=0, 
 b\geq 3\},   \ee
(which includes $S_3$ under the guise of $D_6$). Also 
$$\S(D_{4p})=\S_{\odd}(D_{4p})\cup \S_{\even}(D_{4p}), $$
 where
\be \S_{\odd}(D_{4p})=\{m\equiv 1 \bmod 4\midd p\nmid m, p^3\mid m\}.
\ee
and
\be \label{D4p}  \S_{\even}(D_{4p})=\{2^ap^b m_{2p}\midd  a=4, a\geq 6,\;\; b=0, 
 b\geq 3\}.  \ee
For $G=\mathbb Z_2 \times \mathbb Z_2$, viewed as $D_4:$ 
\be \label{Z2Z2} \S(\Z_2\times\Z_2)=\{4m+1, \;\; 2^4(2m+1),\;\; 2^6m\midd m\in \mathbb Z\}, \ee
for  $G=D_8$:
\be \label{D8}  \S(D_8)=\{4m+1,\;\; 2^8m \midd m\in \mathbb Z\},   \ee
for  $G=D_{16}$:
\be \label{D16}   \S(D_{16})=\{4m+1,\;\; 2^{10}m \midd m\in \mathbb Z\},  \ee
with upper and lower set inclusions for the other $D_{2^k}$. For
$G=D_{2p^2}$ when  $p=3,5$ or $7$:
$$ S(D_{2p^2})=\{2^a p^b m_{2p} \midd  a=0, a\geq 2,\;\; b=0,
b\geq 5\}. $$
Also, the value of $\lambda(D_k)$ was determined for $k<3.79 \times 10^{47}$ in \cite{dihedral}.

Many of the small groups are of one of the above forms. Indeed for the groups of 
order at most 14 this just leaves out the groups
$G=Q_8,\mathbb Z_2 \times\mathbb Z_4$, $\mathbb Z_2^3$, $\mathbb Z_3 \times 
\mathbb Z_3$, $A_4, Q_{12}$, $\mathbb Z_{12}$  and $\mathbb Z_6 \times \mathbb 
Z_2$, where $Q_{4n}$ denotes the dicyclic group of order $4n$.
Our goal here is to determine $\S(G)$ for 
these  remaining groups $G$. As we shall see, for example for $G=\mathbb Z_6\times \mathbb Z_2$, 
these can become complicated 
very quickly, so developing a general theory for dealing with all finite groups is probably not feasible.
While it is known \cite{Formanek}  that the group determinant polynomial, $\mathscr{D}_G(x_{g_1},\ldots ,x_{g_n})$, determines the group, it remains open whether the set of integer values taken, $\S(G)$, also determines the group.

We shall make frequent use of the multiplication property 
\be \label{mult} \mathscr{D}_G(a_{g_1},\ldots 
,a_{g_n})\mathscr{D}_G(b_{g_1},\ldots ,b_{g_n})=\mathscr{D}_G(c_{g_1},\ldots 
,c_{g_n}),\;\; c_g:=\sum_{uv=g}a_ub_v, \ee
corresponding to multiplication $\left(\sum_{g\in G} a_g g 
\right)\left(\sum_{g\in G}  b_g g \right)= \sum_{g\in G}  c_g g$ in 
$\mathbb Z[G]$  (or multiplication and reduction of polynomials subject to the 
relations). Thus $\S(G)$ is a semigroup.

We shall work interchangeably with the group determinants 
$\mathscr{D}_G(x_{g_1},\ldots ,x_{g_n})$ and the polynomial measures 
$M_G(F)$. We begin by expressing the group determinant for the dicyclic group 
$Q_{4n}$ as a $\mathbb Z_{2n}$  Lind measure of an 
associated polynomial.

\section{ Dicyclic Groups}

We write the dicyclic group of order $4n$ in the form 
$$ Q_{4n} =\langle a,b \; :\; a^{2n}=1,b^2=a^n,ab=ba^{-1}\rangle, $$
and order the elements $1,a,a^2,\ldots ,a^{2n-1},b,ba,\ldots ,ba^{2n-1}$.

Our polynomial measures will be defined on $\mathbb Z[x,y]/\langle x^{2n}-1, 
y^2-x^n,xy-yx^{2n-1} \rangle,$ where we can assume that $F$ in $\mathbb Z[x,y]$ 
has been reduced to the form
\be \label{reducedform}  F(x,y)=f(x)+yg(x), \;\;\;\; \; f(x)=\sum_{j=0}^{2n-1} 
a_j x^j,\;\;\; g(x)=\sum_{j=0}^{2n-1} b_j x^j. \ee
The case $n=2$ gives us the classical quaternion group
$$ Q_8=\{1,-1,i,-i,j,-j,k,-k\},\;\;\; i^2=j^2=k^2=ijk=-1, $$
under the correspondence $(1,a,a^2,a^3,b,ba,ba^2,ba^3)=(1,i,-1-i,j,-k,-j,k)$.

Our determinant $\mathscr{D}_{Q_{4n}}(a_0,\ldots ,a_{2n-1},b_0,\ldots 
,b_{2n-1})$ will have four linear factors, corresponding to the characters $\chi 
(a)=1$ and $\chi(b)=\pm 1$,  and $\chi(a)=-1$ with
$\chi(b)=\pm 1$ if $n$ is even and $\chi (b)=\pm i$ if $n$ is odd,
\begin{align*} 
\Big(f(1)+g(1)\Big)\Big(f(1)-g(1)\Big)\Big(f(-1)+g(-1)\Big)\Big(f(-1)-g(-1)\Big)
, & \;\;\; \text{ $n$ even,} \\
\Big(f(1)+g(1)\Big)\Big(f(1)-g(1)\Big)\Big(f(-1)+ig(-1)\Big)\Big(f(-1)-ig(-1)\Big), & \;\;\; \text{ $n$ odd.} 
\end{align*}
For the remaining complex $2n$th roots of unity  $\om=\om_{2n}^j$, $1\leq j \leq 
n-1$,
 (complex conjugates give the same factors) we have $n-1$ two-dimensional 
representations
$$    \rho (a)=\left( \begin{matrix} \om & 0 \\ 0 & 
\om^{-1}\end{matrix}\right),\;\;\;\rho(b) = \left( \begin{matrix} 0 & \om^n \\ 1 
& 0 \end{matrix} \right),    $$
leading to  the squares of $n-1$ quadratic factors
\begin{align*}  
\det\left( \rho\left( \sum_{j=0}^{2n-1}  a_j a^j + \sum_{j=0}^{2n-1} b_j  b 
a^j\right)\right)  
 & = \det \left( \begin{matrix} f(\omega) & \om^n g(\om^{-1}) \\ g(\om) & f(\om^{-1}) 
\end{matrix} \right) \\
 & =f(\om)f(\om^{-1}) -\om^n g(\om)g(\om^{-1}). \end{align*}
Hence we can write 
$$ \mathscr{D}_{Q_{4n}}(a_0,\ldots ,a_{2n-1},b_0,\ldots 
,b_{2n-1})=\prod_{j=0}^{2n-1} (f(\om_{2n}^j)f(\om_{2n}^{-j}) 
-\om_{2n}^{jn}g(\om_{2n}^j)g(\om_{2n}^{-j}). $$
We take this to be the dicyclic  measure of an $F(x,y)$ in $\mathbb Z[x,y]$, 
reduced to the form  \eqref{reducedform},
\be\label{E-Q} M_{Q_{4n}}(F) = M_{\mathbb Z_{2n}}\Big( f(x)f(x^{-1})-x^n g(x)g(x^{-1})\Big). 
\ee
We observe for future reference that
\begin{align}\label{E-ggff}
M_{\mathbb Z_{2n}}\Big( g(x)g(x^{-1})-x^n& f(x)f(x^{-1})\Big) \nonumber\\ 
&=(-1)^nM_{\mathbb Z_{2n}}\Big( f(x)f(x^{-1})-x^n g(x)g(x^{-1})\Big),
\end{align}
so that $\S(Q_{4n})=-\S(Q_{4n})$ when $n$ is odd.

\section{Groups of order 8}\label{8gps}

In this section we determine $\S(G)$ for the 
five  groups of order eight: $G= \mathbb Z_8$, $D_8$,  $Q_8,$ $\mathbb Z_4 \times\mathbb 
Z_2$ and $\Z_2^3$.

As mentioned in the introduction, $\S(G)$ is already known for $G=\mathbb 
Z_8$ and $D_8$, namely
 $$ \S(\Z_8)=\{2m+1 \text{ and } 32m \midd m\in \mathbb Z\} $$
and 
$$\S(D_8)= \{4m+1  \text{ and } 2^8m \midd  m\in \mathbb Z\}. $$

For the groups 
 $G=D_8, Q_8$ and $\mathbb Z_4 \times \mathbb Z_2$, the  group determinants 
correspond to  Lind-Mahler measures on the two-variable polynomials $F(x,y)\in 
\mathbb Z[x,y]$  reduced to 
$$ F(x,y)=f(x)+yg(x), \text{ where } f(x) = \sum_{j=0}^3 a_jx^j,\;\; g(x)=\sum_{j=0}^3 
b_jx^j. $$
These are given in terms of Lind-Mahler measures of cyclic groups by \eqref{E-D} for $D_8$,
by \eqref{E-Q} for $Q_8$, and by
\be\label{E-Z42}
M_{\Z_4\times\Z_2}(F)=M_{\Z_4}(f(x)+g(x))\cdot M_{\Z_4}(f(x)-g(x))
\ee
for $\Z_4\times\Z_2$.

For $G=\mathbb Z_2^3$ the determinants correspond to measures of polynomials in 
$\mathbb Z[x,y,z]$, reducible mod $\langle x^2-1,y^2-1,z^2-1\rangle$  to 
$$ F(x,y,z)=\sum_{i,j,k\in \{0,1\}} a_{i,j,k}x^iy^jz^k \in \mathbb Z[x,y,z],\;\;\; M_G(F)=\prod_{x,y,z=\pm 1} F(x,y,z). $$

\begin{theorem} \label{Q8} We have
$$\S(\Z_4\times\Z_2)= \{8m+1  \text{ and } 2^8m \midd  m\in \mathbb Z\},$$
$$\S(Q_8)= \S(\Z_4\times\Z_2)\cup \{(8m-3)p^2\midd m\in\Z, \; p \equiv 3\bmod 4 \text{ prime}\},$$
and
$$\S(\Z_2^3)=\{8m+1 \text{ and }2^8(4m+1)  \text{ and } 2^{12}m \midd m\in\Z\}. $$ 
\end{theorem}

Note  that  
$$\S(\Z_2^3)\subsetneq \S(\Z_4\times\Z_2) \subsetneq \S(Q_8) \subsetneq\S(D_8)\subsetneq \S(\Z_8). $$ 

The Theorem immediately gives us the minimum non-trivial measure for the groups of order 8. 
\begin{corollary}
$$\lambda(\mathbb Z_4\times \mathbb Z_2)=\lambda(\mathbb Z_2^3)=\lambda(Q_8)=7 \text{ and }  \lambda (D_8)=\lambda(\mathbb Z_8)=3. 
$$
\end{corollary}

\section{The Alternating Group $A_4$}\label{S-A4}

Taking the two generators
$$ \alpha =(123),\;\;\; \beta=(12)(34),  $$
we order the elements $(g_1,g_2,\ldots ,g_{12})$ of $A_4$ as
\begin{align*}
& ( 1, (12)(34), (13)(24), (14)(23), (123), (243), (142), (134), (132),(143),  
(234), (124) )\\
 & \;\;\;\; =(1, \beta, \alpha^2\beta \alpha , \alpha \beta \alpha^2 ,\alpha 
,\beta \alpha,\alpha^2 \beta \alpha^2 ,\alpha \beta, \alpha^2, \beta \alpha^2, 
\alpha^2 \beta ,\alpha \beta \alpha).
\end{align*}
Now $A_4$ has  four irreducible representations: three linear ones 
$\chi_0,\chi_1,\chi_2$ where $\chi(\beta)=1$ and $\chi(\alpha)=1,\om$ or $\om^2$ 
respectively with $\om :=e^{2\pi i/3}$, and one, $\rho$,  of degree $3$. 
This latter representation comes from isometries of the regular 
tetrahedron (vertices $1,2,3,4$) relative to the three axes passing through the 
centres of opposite pairs of sides. Explicitly, 
\[
\rho(\alpha) =\left(\begin{matrix}0&0&1\\1&0&0\\0&1&0\end{matrix}\right)
\qquad\text{ and } \qquad
\rho(\beta) = \left(\begin{matrix} 1&0&0\\0&-1&0\\0&0&-1\end{matrix}\right),
\]
generating the representation:
\begin{align*}  
\rho( \alpha^2 \beta \alpha) =  
\left(\begin{matrix}-1&0&0\\0&-1&0\\0&0&1\end{matrix}\right), &\;\;\;
\rho( \alpha \beta \alpha^2) = 
\left(\begin{matrix}-1&0&0\\0&1&0\\0&0&-1\end{matrix}\right),\;\;\;
\rho( \beta \alpha) = 
\left(\begin{matrix}0&0&1\\-1&0&0\\0&-1&0\end{matrix}\right),\;\;\; \\
\rho( \alpha^2\beta \alpha^2) = 
\left(\begin{matrix}0&0&-1\\-1&0&0\\0&1&0\end{matrix}\right), & \;\;\;
\rho(\alpha \beta) = 
\left(\begin{matrix}0&0&-1\\1&0&0\\0&-1&0\end{matrix}\right),\;\;\; 
 \rho( \alpha^2) = 
\left(\begin{matrix}0&1&0\\0&0&1\\1&0&0\end{matrix}\right),\;\;\; \\
\rho( \beta \alpha^2) = 
\left(\begin{matrix}0&1&0\\0&0&-1\\-1&0&0\end{matrix}\right),& \;\;\;
\rho( \alpha^2 \beta)=  
\left(\begin{matrix}0&-1&0\\0&0&-1\\1&0&0\end{matrix}\right),\;\;\;
\rho( \alpha \beta \alpha)=  
\left(\begin{matrix}0&-1&0\\0&0&1\\-1&0&0\end{matrix}\right).\;\;\;
\end{align*}
Writing $x=\sum_i a_i g_i$ in $\mathbb Z [G]$,  
then for $G=A_4$ the group determinant takes the form
$$ \mathscr{D}_{G}(a_1,a_2,\ldots ,a_{12}) =l_0 l_1 l_2 D^3 $$
where, putting
$$a:=a_1+a_2+a_3+a_4,\;\;\; b:=a_5+a_6+a_7+a_8 \;\; \text{ and } \;\; 
c:=a_9+a_{10}+a_{11}+a_{12},$$ 
we have
$$ l_0 =\chi_0(x)=a+b+c,\;\;\;l_1=\chi_1(x) = a+b\om +c\om^2,\;\;\;l_2= 
\chi_2(x)=a+b\om^2+c\om, $$
and
$$ D=\det(\rho(x))=\det\left( \begin{matrix} 
a_1+a_2-a_3-a_4&a_9+a_{10}-a_{11}-a_{12}&a_5+a_6-a_7-a_8\\
                     a_5-a_6-a_7+a_8& a_1-a_2-a_3+a_4&a_9-a_{10}-a_{11}+a_{12}\\
a_9-a_{10}+a_{11}-a_{12}&a_5-a_6+a_7-a_8&a_1-a_2+a_3-a_4
\end{matrix}\right). $$

We  can regard $\mathscr{D}_{G}(a_1,a_2,\ldots ,a_{12})$ as the Lind measure 
$M_G(F)$ of the generic polynomial
\begin{align*}  F(x,y) = a_1  +  & a_2 y + a_3 x^2 y x+ a_4 x y x^2 + a_5 x+ a_6 
yx +a_7 x^2 y x^2 \\
   &  + a_8 xy + a_9 x^2 +a_{10} y x^2 + a_{11} x^2 y + a_{12}x y x, 
\end{align*}
in $\mathbb Z[x,y]$ with non-commutative multiplication,  and reduction according to the 
relations 
$$x^3=1,\;\; y^2=1,\;\; yxy=x^2yx^2\;\; \text{ and }\;\; yx^2y=xyx. $$

\begin{theorem} \label{A4} We have $\S(A_4)=\S(A_4)_{\even}\cup\S(A_4)_{\odd}$, where
\[
\S(A_4)_{\even}  =\{2^a 3^b m_6\midd  a=4, a\geq 8, \; b=0, b\geq 2\}
\]
and
\[
\S(A_4)_{\odd}=\{m\equiv 1 \bmod 4 \midd 3\nmid m, 3^2\mid m\}.
\]
\end{theorem}

\section{Groups of Order 12} \label{12gps}

There are five groups of order twelve: $\mathbb Z_{12},$ $\mathbb Z_6 \times 
\mathbb Z_2$, $D_{12}$, $Q_{12}$ and $A_4$ (dealt with in the 
previous section). 

From \cite{dihedral} we know that
$\S(D_{12})=\S(D_{12})_{\even}\cup\S(D_{12})_{\odd}$, where
\[
\S(D_{12})_{\even}=\{ 2^a3^bm_6 \midd   a=4, a\geq 6,\;\;  b=0, b\geq 3\}
\]
and
\[
\S(D_{12})_{\odd} =\{m\equiv 1\bmod 4 \midd  3\nmid m, 3^3\mid m\}.
\]

For the groups $G=Q_{12}$, $D_{12}$ and $\mathbb Z_6 \times \mathbb Z_2$ we  
work with measures of polynomials
$$ F(x,y)=f(x)+yg(x)\in \mathbb Z[x,y],\;\;\; f(x)=\sum_{j=0}^5 a_jx^j,\;\; 
g(x)=\sum_{j=0}^5 b_jx^j. $$
These are given in terms of Lind-Mahler measures of cyclic groups by \eqref{E-Q} for $Q_{12}$,
by \eqref{E-D} for $D_{12}$, and by
\be\label{E-Z62}
M_{\Z_6\times\Z_2}(F)=M_{\Z_6}(f(x)+g(x))\cdot M_{\Z_6}(f(x)-g(x)).
\ee

\begin{theorem} \label{Q12}
If  $G=Q_{12}$ the set $\S(Q_{12})$ consists of measures  $M_G(F)$ of the following forms
\begin{align} 2^a 3^b m_6 \midd \; &  a=0, 4, a\geq 
6,\;\; b=0,  b\geq 3, \\
2^5 3^b m_6 \midd \; &  b=4, b\geq 6,\\
\label{32m} 2^5 3^bm_6k \midd\; &   b=0, 3, 5, 
 \end{align}
where, in \eqref{32m},  $k$ can be a prime $p\equiv 5\bmod 12$ and also the square of a  prime $p\equiv 5 \bmod 6$.

\comments{\vspace{1ex}
\noindent
For $G=D_{12}$ the odd measures are the $m\equiv 1 \bmod 4$ with $3\nmid m$ or 
$27\mid m$.
The even measures take the form $2^4m$ with $2\nmid m$ or $2^6m$ with $3\nmid m$ 
or $3^3\mid m$.
}

\end{theorem}

For $G=\mathbb Z_6\times\mathbb Z_2$ and $\mathbb Z_{12}$ we need to partition the 
primes $p\equiv 1\bmod 12$ into two sets
\begin{align}
\P_1:& =\{p\equiv 1 \bmod 12 \; :\; p=(6k+2)^2 + (6t+3)^2 \text{ for some 
$k,t\in \mathbb Z$}\},\nonumber \\
\label{defS} \P_2:& =\{p\equiv 1 \bmod 12 \; :\; p=(6k)^2 + (6t+1)^2 
\text{ for some $k,t\in \mathbb Z$}\}. 
\end{align}

These sets are disjoint by the uniqueness of representation of primes $\equiv 1\bmod 4$ as a sum of two squares. They are probably not describable  by a simple congruence; see  Cox \cite{Cox} for a  
class field theory approach  to distinguishing which primes are  of the form 
$x^2+36y^2$.

\vspace{2ex}
\noindent
\begin{theorem}\label{Z2Z6}  For   $G=\mathbb Z_6 \times \mathbb Z_2$, the set $\S(\mathbb Z_6 \times \mathbb Z_2)$ consists of measures $M_G(F)$ of the following forms:

\vspace{1ex}

\noindent
(a) The measures with $3^3\mid M_G(F)$ take the form
$$ 3^3(4m-1),\;\;\; 2^4\cdot 3^3 (2m-1),  
  \;\;\;2^6\cdot 3^3 m. $$

\vspace{1ex}

\noindent
(b) The measures with $3^2 \parallel  M_G(F)$ take the form
$$ 3^2(4m-1)p\;\;\; 2^4 \cdot 3^2(2m-1)p  \;\;\;  2^6 \cdot 3^2 mp  $$
for some prime $p\equiv 7 \bmod 12$, or 
$$ 2^8 \cdot 3^2 (4m-1) \;\;\; \hbox{ or } \;\;\; 2^{10}\cdot 3^2 (4m-1) \;\;\; 
\hbox{ or }\;\;\; 2^{12} \cdot 3^2 (2m-1) \;\;\; \hbox{ or } \;\;\; 2^{14}\cdot 
3^2 m. $$

\vspace{1ex}

\noindent
(c) The measures coprime to $3$ take the form
$$ 12m+1, \;\;\;  2^4(6m+1)  \;\; \; 2^6(3m+1), $$
or
\be \label{type2}  (12m+5)k,  \;\;\; -2^4(6m+1)k,\;\;\; -2^6(3m+1)k \ee
where $k=p$ for  some prime $p\equiv 1\bmod 12$ in $\P_1$, or $k=p^2$ for some 
$p\equiv 5\bmod 12$, or $k=p_1p_2$ for some primes $p_1,p_2\equiv 7\bmod 12$, or
$$ 2^8(12m+5), \;\;\;  2^8(12m+5)p, \;\;\;   
2^{10}(12m+5),  \;\;\;  2^{10}(12m+5)p, $$
for some $p\equiv 7\bmod 12$, or
$$-2^{12}(6m+1), \;\;\;  -2^{14}(3m+1).$$
In each case $m$ runs through all the integers.

\end{theorem}
Note that in the theorem there are no measures with $3\parallel M_G(F)$.

For the group $G=\mathbb Z_{12}$, we work with measures on polynomials 
$F(x)=\sum_{j=0}^{11} a_jx^j$.

\begin{theorem}\label{Z12}
Let $G=\mathbb Z_{12}$.  We separate  $\S(\Z_{12})$ into the odd and even  measures:

\vspace{1ex}
\noindent
(a) The odd measures coprime to $3$ can take any value $m_6$.

\noindent
The odd multiples of $3$ take the form
$9m_6p$ for  primes $p\equiv 5$ and $7\bmod 12$, and $p$ in $\P_1$, and $27m_2$.

\vspace{1ex}
\noindent
(b) The even values divisible by $3$  take the form $2^4 \cdot 3^2 m$ for all 
integers $m$.

\noindent
The even values coprime to $3$ take the forms $2^4m_6$, and
$2^5m_6p$ for some prime  $p\equiv 5$ and $7$ mod $12,$ and $p$  in $\P_1$, and $2^6m_3$.

\end{theorem}
Again, in this theorem there are no measures with $3\parallel M_G(F)$.

\

The Lind-Lehmer constant (the minimal non-trivial measure)  is readily 
determined for the  groups of order 12:
\begin{corollary}
$$\lambda(D_{12})=\lambda(Q_{12})=\lambda(\mathbb Z_{12})=\lambda(A_4)=5,\;\;  \lambda(\mathbb 
Z_6\times\mathbb Z_2)=11
$$
\end{corollary}

\section{The remaining group $\mathbb Z_3 \times \mathbb Z_3$} 
\label{oddments}

\begin{theorem}\label{Z3Z3} For $G=\mathbb Z_3 \times \mathbb Z_3$ we have
$$ \S(G)= \{9m\pm 1 \midd m\in\Z\}\cup \{3^6m  \midd m\in\Z\}. $$
\end{theorem}

\section{Proofs for Section \ref{8gps}}

\begin{proof}[Proof of Theorem \ref{Q8}] We first prove the result for $G=Q_8$ and $\Z_4\times\Z_2$. Since it requires no extra work
we also include $D_8$, although it is already covered in \cite{dihedral}. We begin by comparing the  form of  the Lind measure
for  these three groups.
In all three cases we have the same four linear factors in $\mathbb 
Z[a_0,a_1,a_2,a_3,b_0,b_1,b_2,b_3]$, namely
\begin{align*}
\ell_1 & :=F(1,1)=(a_0+a_2)+(a_1+a_3) +(b_0+b_2)+(b_1+b_3),\\
\ell_2 & :=F(1,-1)=(a_0+a_2)+(a_1+a_3) - (b_0+b_2)- (b_1+b_3),\\
\ell_3 & :=F(-1,1)=(a_0+a_2)-(a_1+a_3) + (b_0+b_2)-(b_1+b_3),\\
\ell_3 & :=F(-1,-1)=(a_0+a_2)-(a_1+a_3) -(b_0-b_2)+(b_1+b_3).\end{align*}
The remaining factors are  quadratics   in $\mathbb 
Z[a_0,a_1,a_2,a_3,b_0,b_1,b_2,b_3]$: 
\begin{align*} \mathscr{D}_{Q_8}(a_0,a_1,a_2,a_3,b_0,b_1,b_2,b_3) & = \ell_1 
\ell_2\ell_3\ell_4 q_1^2, \\
\mathscr{D}_{D_8}(a_0,a_1,a_2,a_3,b_0,b_1,b_2,b_3)&= \ell_1 \ell_2\ell_3\ell_4 
q_2^2,\\
\mathscr{D}_{\mathbb Z_4\times \mathbb Z_2}(a_0,a_1,a_2,a_3,b_0,b_1,b_2,b_3)&= 
\ell_1 \ell_2\ell_3\ell_4 q_3q_4,
\end{align*}
where
\begin{align*}
 q_1 & :=|f(i)|^2+|g(i)|^2=(a_0-a_2)^2+(a_1-a_3)^2+(b_0-b_2)^2+(b_1-b_3)^2,\\
 q_2 & :=|f(i)|^2-|g(i)|^2=(a_0-a_2)^2+(a_1-a_3)^2-(b_0-b_2)^2-(b_1-b_3)^2,\\
q_3  & :=|f(i)+g(i)|^2= ((a_0-a_2)+(b_0-b_2))^2+ ((a_1-a_3)+(b_1-b_3))^2,   \\
q_4 & := |f(i)-g(i)|^2=  ((a_0-a_2) -(b_0-b_2))^2 +((a_1-a_3) -(b_1-b_3))^2 .
\end{align*}

Thus for $G=Q_8$, $D_8$ or $\mathbb Z_4 \times \mathbb Z_2$ we have
\begin{align*} 
&\mathscr{D}_{G}(m+1,m,m,m,m,m,m,m) =8m+1,\\
&\mathscr{D}_{G}(k+ 2,k,k,k,k,k,k,k) =2^8(4k+1),\\
&\mathscr{D}_{G}(k-1,k+1,k-1,k+1,k+1,k+1,k,k) =-2^8(4k+1),\\
&\mathscr{D}_{G}(k+1,k,k+1,k,k-1,k-1,k,k)  =2^8(2k).
\end{align*}
Equivalently, writing these as polynomial measures we have
\begin{align*}  & M_G\left( 1+ m\frac{x^4-1}{x-1} + 
ym\frac{x^4-1}{x-1}\right)=8m+1,\\
& M_G\left( 2+ k\frac{x^4-1}{x-1} + yk\frac{x^4-1}{x-1}\right)=2^8(4k+1),\\
 & M_G\left( (x^2+1)(x-1)+ k\frac{x^4-1}{x-1} + y\left( 
(x+1)+k\frac{x^4-1}{x-1}\right)\right)=-2^8(4k+1),\\
 & M_G\left( (x^2+1)+ k\frac{x^4-1}{x-1} + 
y\left(-(x+1)+k\frac{x^4-1}{x-1}\right)\right)=2^8(2k).
\end{align*}
Writing $p\equiv 3$ mod 4 as $p=a^2+b^2+c^2+d^2$ with $a$ even, $b,c,d$ odd,  
then
\begin{align*}
(a_0,a_1,a_2,a_3) & 
=\left(m+\frac{a}{2},m+\frac{(b-1)}{2},m-\frac{a}{2},m-\frac{(b+1)}{2}\right),\\
 (b_0,b_1,b_2,b_3) & 
=\left(m+\frac{(c-1)}{2},m+\frac{(d-1)}{2},m-\frac{(c+1)}{2},m-\frac{(d+1)}{2}
\right),  
\end{align*}
has $\mathscr{D}_{Q_8}(a_0,a_1,a_2,a_3,b_0,b_1,b_2,b_3)=(8m-3)p^2,$
while
$$\mathscr{D}_{D_8}(m,m,m,m-1,m,m,m-1,m-1)=8m-3. $$

It remains to show that a determinant $\mathscr{D}_G$ takes one of the stated 
forms. Suppose first that $\mathscr{D}_G$ is even. Since the $\ell_j\equiv 
\ell_1$ mod 2 and the $q_i\equiv \ell_1^2$ mod 2 we know that the $\ell_j$ and 
$q_j$ are all even. If $(a_0-a_2)$, $(a_1-a_3)$, $(b_0-b_2)$, $(b_1-b_3)$ 
are all even  or all odd then  in all cases $4\mid q_j$ and $2\mid \ell_j$ and 
$2^8\mid \mathscr{D}_G$.
So suppose two of them are even and two odd. Hence two of $(a_0+a_2)$, 
$(a_1+a_3)$, $(b_0+b_2)$, $(b_1+b_3)$ 
are even and two odd. Call these $A,B,C,D,$ in any order,  then 
$$ \ell_1\ell_2\ell_3\ell_4=((A+B)^2-(C+D)^2)((A-B)^2-(C-D)^2). $$
Hence if $A,C$ are odd and $B,D$ even then $(A\pm B)^2-(C\pm D)^2 \equiv 1-1=0$ 
mod 8, so $2^6\mid  \ell_1\ell_2\ell_3\ell_4$
and $2\mid q_j$, and $2^8\mid \mathscr{D}_G$.

Suppose that $\mathscr{D}_G$ is odd. So either  one or three of the  
$(a_0-a_2)$, $(a_1-a_3)$, $(b_0-b_2)$, $(b_1-b_3),$ 
will be odd (and hence one or three  of  the corresponding  $A,B,C,D$,  will 
odd). Suppose that $A$ is odd and $B$ is even and $C,D$ have the same parity.  
Plainly $q_1^2,q_2^2\equiv $ 1 mod 8 and
$$ q_4=q_3 -4(a_0-a_2)(b_0-b_2) -4(a_1-a_3)(b_1-b_3).   $$
So if three of $A,B,C,D$ are  even  then  $q_4\equiv q_3$ mod 8 and 
$q_3q_4\equiv q_3^2\equiv 1$ mod 8, while if three are odd
we have  $q_4\equiv q_3+4$ mod 8 and $q_3q_4\equiv q_3^2-4\equiv -3$ mod 8.
Similarly
\begin{align*} (A-B)^2 - (C-D)^2  & = (A+B)^2 - (C+D)^2 -4AB+4CD \\
& \equiv (A+B)^2 - (C+D)^2 +4CD \text{ mod 8}.\end{align*}
Hence if $C,D$ are even we have $\ell_1\ell_2\ell_3\ell_4 \equiv 
((A+B)^2-(C+D)^2)^2 \equiv 1$ mod 8 and 
$\mathscr{D}_G\equiv 1$  mod 8.
If $C,D$ are odd then $\ell_1\ell_2\ell_3\ell_4 \equiv ((A+B)^2-(C+D)^2)^2-4 
\equiv -3$ mod 8, and
  $\mathscr{D}_G\equiv 1$ mod 8 for $G=\mathbb Z_4 \times \mathbb Z_2$, and  
$\mathscr{D}_G\equiv -3$ mod 8 for $G=Q_8$ or $D_8$, with $q_1\equiv 3$ mod 4 
for $G=Q_8$ (and so divisible by at least one $p\equiv 3$ mod 4).
This completes the proof for $G=Q_8$, $\Z_4\times\Z_2$ and $D_8$.

Next, for $G=\mathbb Z_2^3$ we have 
$\displaystyle M_G(F)=\prod_{x,y,z=\pm 1} F(x,y,z)$. 

We can achieve anything of the form $8m+1$, $2^8(4m+1)$ or $2^{12}m$ using
\begin{align*}
& M_G\left( 1+m(1+x)(1+y)(1+z)\right)  = 8m+1, \\
& M_G\left( 2+m(1+x)(1+y)(1+z)\right)  = 2^8(4m+1), \\
&M_G\left( 3+z+k(1+x)(1+y)(1+z)\right)  = 2^{12}(2k+1), \\
& M_G\left( x+y+z-3+(1-x)(1-y)(1-z) +k(1+x)(1+y)(1+z)\right)  = 2^{12}(2k), 
\end{align*}
so it just remains to check that any $M_G(F)$ is  of one of these forms.

We write $F(x,y)=f(x,y)+zg(x,y)$, with $f(x,y)$ and $g(x,y)$ of the form
$$ g(x.y)=a(0,0) + a(1,0)x+a(0,1)y+a(1,1)xy\in \mathbb Z[x,y], $$
so that 
$$ M_G(F) = \prod_{x,y=\pm 1} f(x,y)^2-g(x,y)^2. $$
Notice that the $f(\pm 1,\pm 1)\equiv f(1,1)$ mod 2, and $g(\pm 1,\pm 1)\equiv 
g(1,1)$ mod 2, and $M_G(F)\equiv (f(1,1)-g(1,1))^4 $ mod 2 is even if $f(1,1)$ 
and $g(1,1)$ have the same parity and odd otherwise.

Suppose first that $M_G(F)$ is odd. Reversing the roles of $f(x,y)$ and $g(x,y)$ 
as necessary we suppose that the 
$f(\pm 1,\pm 1)$ are all  odd and the $g(\pm 1,\pm 1)$ all even.  Then mod 8 we 
have 
\begin{align*}  M_G(F)  & \equiv \prod_{x,y=\pm 1} (1-g(x,y)^2 )
  \equiv 1- \sum_{x,y=\pm 1} g(x,y)^2\equiv 1- \left(\sum_{x,y=\pm 1} 
g(x,y)\right)^2.
\end{align*}
But 
$$\sum_{x=\pm 1} g(x,y)=4a(0,0), $$
and  so $M_G(F)\equiv 1$ mod 8.

Suppose that $M_G(F)$ is even. If  the $f(\pm 1,\pm 1)$ and $g(\pm 1,\pm 1)$ are 
all odd then the $f(x,y)^2-g(x,y)^2\equiv 1-1=0$ mod $8$ for each of the four 
factors and $2^{12} \mid M_G(F)$. So suppose that they are all even and
$$ M_G(F) = 2^8 \prod_{x,y=\pm 1} (f(x,y)/2)^2 - (g(x,y)/2)^2. $$
If any of the $f(x,y)/2$ and $g(x,y)/2$ have the same parity then 4 divides that 
factor and $2^{12}\mid M_G(F)$. So assume that they have opposite  parity for 
all $x,y=\pm 1,$ with $(f(x,y)/2)^2 - (g(x,y)/2)^2$ equalling  $1$ mod 4 if 
$g(x,y)/2$ is even and $-1$ mod 4 if $g(x,y)/2$ is odd. But the $\sum_{x,y=\pm 
1} g(x,y)/2=2a(0,0)$ is even, so we must have an even number of $-1$'s and 
$\prod_{x,y=\pm 1} (f(x,y)/2)^2 - (g(x,y)/2)^2\equiv 1$ mod 4.
\end{proof}

\section{Proof of Theorem  \ref{A4}}

Suppose that 3 divides
$$\mathscr{D}_G(a_1,\ldots ,a_{12})=l_0l_1l_2 D^3. $$ 
If $3\mid l_0 l_1 l_2$ then 3 divides  $l_0$ or $l_1l_2$
and, from  the  congruence
$$ l_1 l_2 \equiv l_0^2 \text{ mod } 3, $$
must divide  both, and $3^2 \mid \mathscr{D}_G(a_1,\ldots ,a_{12})$.
If $3\mid D$ then plainly $3^3\mid \mathscr{D}_G(a_1,\ldots ,a_{12})$.  Hence 
$3\mid \mathscr{D}_G(a_1,\ldots ,a_{12})$
implies that $3^2\mid \mathscr{D}_G(a_1,\ldots ,a_{12})$.

It is easy to see that $D\equiv \det \begin{pmatrix}  a & c & b\\ b & a & c \\ c 
& b & a   \end{pmatrix}$ mod 2, the circulant determinant equalling $l_0l_1l_2$, 
the $\mathbb Z_3$ measure of $a+bx+cx^2$,  but in fact expanding  we have the 
stronger congruence
$$ D\equiv l_0l_1l_2 \text{ mod } 4. $$
From this we see that any odd determinant must have $l_0l_1l_2$ odd and
$$ \mathscr{D}_G(a_1,\ldots ,a_{12})\equiv (l_0l_1l_2)^4 \equiv 1 \text{ mod } 
4. $$
If the determinant is even then 2 divides $l_0l_1l_2$ or $D$ and from the 
congruence must divide both. Moreover if $2\parallel  l_0 l_1 l_2$
then $2\parallel D$  and $2^4 \parallel  \mathscr{D}_G(a_1,\ldots ,a_{12})$, 
while if $2^2\mid l_0l_1l_2$ then $2^2\mid D$ and $2^8\mid  
\mathscr{D}_G(a_1,\ldots ,a_{12})$.

Hence the determinants must be of the stated form. It remains to show that we 
can achieve all these.
From
\begin{align*}
\mathscr{D}_G(1,1,0,0,0,0,0,0, 1,0,0,0) & =9,\\
\mathscr{D}_G(1,1,-1,0, 0,0,0,0, 0,0,0,0) & = -27,
\end{align*}
and  multiplication we can obtain any $\pm 3^b$ with $b\geq 2$ which is $1$ mod 
$4$.

For the powers of 2 we have 
\begin{align*}
 &\mathscr{D}_G(0,0,0,0, 1,0,0,0, 1,0,0,0)  = 2^4,\\
& \mathscr{D}_G(1,-1,0,0, 1,0,0,0, 1,0,0,0)  = -2^4,
\end{align*}
and
\begin{align*}
& \mathscr{D}_G(k+2,k,k,k, k+1,k,k,k, k+1,k,k,k)  =2^8(1+3k),\\
& \mathscr{D}_G(k+2,k+1, k-1,k,  k+1,k,k,k, k+1,k,k,k)  =-2^8(1+3k),
\end{align*}
with  $k=0,-1,1,-3$ giving $\pm 2^8,\pm 2^9,\pm 2^{10}, \pm 2^{11}.$ 
Multiplication of these gives all  $\pm 2^a$ with $a=4$ or $a\geq 8$.

The $m\equiv 1$ mod 4 with $(m,6)=1$ can be obtained with
\begin{align*}
& \mathscr{D}_{G}(k+1, k,k,k, k,k,k,k ,k,k,k,k)=1+12k, \\
& \mathscr{D}_{G} (k+1,k,k,k, k+1,k+1,k,k, k+1,k+1,k,k)=5+12k.
\end{align*}
Multiplication of these produces  all  the forms in Theorem \ref{A4}. \qed

\section{Proofs for Section \ref{12gps}}
We write $\om$ for the primitive cube root of unity  $\om_3 =e^{2\pi i/3}$, and observe that  $\om i$ is a primitive
$12$th root of unity.

We shall need some results on factoring in $\mathbb Z [i]$
and $ \mathbb Z[\om]$ and $\mathbb Z[\om_{12}]=\mathbb Z[\om i]=\mathbb Z[\om,i]$. Note that  (see for 
example \cite[Chapter 11]{Washington})  all these rings are UFDs, with the primes 
splitting in $\mathbb Z[i]$ being $2$ and those $p\equiv 1$ mod 4 and in 
$\mathbb Z[\om ]$ being $3$ and  those primes $p$ with  
$\left(\frac{-3}{p}\right)=1$, 
namely  $p\equiv 1$ or $7\bmod 12$. The primes  $p\equiv 1\bmod 12$ split in both 
$\mathbb Z[\om]$ and $\mathbb Z[i]$, being a product of $4$ primes in 
$\mathbb Z [\om,i]$.
We write $N_1$ and $N_2$ for the norms for $\mathbb Z[\om]$ and $\mathbb 
Z[\om,i]$, so that
\begin{align*} N_1(H(\om))  & =H(\om)H(\om^2)=|H(\om)|^2,\\
 N_2(H(\om,i)) & =H(\om,i)H(\om^2,i)H(\om,-i)H(\om^2,-i).
\end{align*}

\begin{lemma}\label{normform1}
If $\alpha$ in $\mathbb Z [\om]$ has $3\nmid N_1(\alpha)$  then one of $\pm 
\alpha$ takes  the form 
$$-1+(A+B\om)(1-\om), \;\;\; A,B\in \mathbb Z,$$
while if  $\gcd(N_1(\alpha),6)=1$ then exactly one of $\pm \alpha,\pm \alpha \om, \pm 
\alpha \om^2$ 
takes the form
$$   - 1 + 2(A+B\om)(1-\om),\;\;\; A,B\in \mathbb Z. $$
\end{lemma}

\begin{proof} Suppose that $3\nmid N_1(\alpha)$. Writing $\alpha=a+b(1-\om)$ we 
have $3\nmid a,$ else $3\mid N_1(\alpha)=a^2+3ab+3b^2$. 
Replacing $\alpha $ by $-\alpha$ as necessary we can assume that $a\equiv -1$ 
mod 3 and, writing $a=-1+3k$, $3=(1-\om)(2+\om)$  
we get
$$\alpha = -1  +(A+B\om)(1-\om).    $$
Suppose also that $2\nmid N_1(\alpha)$. We can't have $A$ even and $B$ odd, 
else 
$$\alpha = 2w+(A+(B-1)\omega)(1-\omega) \; \Rightarrow \; 2\mid N_1(\alpha). $$

 If $A,B$ are both odd we take
$$ \alpha \om= -1+ ((1-B) + (A-B)\om)(1-\om), $$
and if $A$ is odd and $B$ even we can take
$$ \alpha \om^2 =-1 + ( (B-A +1)+ (1-A)\om)(1-\om). $$
With $A,B$ even the $\alpha \om^j$   cycle through the three possible parity 
combinations.
\end{proof}

In \eqref{defS} we partitioned the primes $p\equiv 1\bmod 12$  into two 
sets $\P_1$ and $\P_2$.
\begin{lemma}\label{normform}
If $p\equiv 7\bmod 12$ then $p=N_1(\alpha)$ for an
$$ \alpha = -1 + 2(1-\om) (2A+1 + B\om), \;\;\;\; A,B \in \mathbb Z, $$
with $B$ even (also  one  with $B$  odd).

If $p\equiv 1\bmod 12$ then $p=N_1(\alpha)$ for an
\be \label{norm1mod12}  \alpha = -1 + 2(1-\om) (2A + B\om), \;\;\;\; A,B \in 
\mathbb Z, \ee
with $B$ even when $p\in\P_2$ and $B$ odd when $p\in\P_1$.

If $p\equiv 1\bmod 12$ then  $p=N_2(\alpha_1)$ for some
\be \label{form1}  \alpha_1 = -1 + (2A+2B\om)(1-\om) + 
i\left(2+(C+2D\om)(1-\om)\right),\;\;\; A,B,C,D\in \mathbb Z, \ee
with  $C$ even when $p\in \P_1$ and $C$ odd when $p$ is in $\P_2$. Also  
$p=N_2(\alpha_2)$ for some
\be \label{2ndform}  \alpha_2= (A+1+2B\om)(1-\om) +  i \left(1+(C+D\om)(1-\om)\right), \;\;\; 
A,B,C,D\in \mathbb Z,\ee
with  $A,C$ and $D$ all  even when $p\in \P_1$ and all  odd when $p\in\P_2$. 
\end{lemma}

\begin{proof}
As above for $p\equiv 7\bmod 12$  we can take $\alpha = -1 + 2(a+b\om)(1-\om)$. 
Note that $a$ must be odd, since otherwise 
$\alpha \equiv -1 +2b\om(1-\om)\equiv -1+2b \text{  mod  } 4$ and 
$N(\alpha)\equiv 1$ mod 4. If $b$ is also odd  then we can replace 
$\alpha$ by its conjugate $-1+2(a+b\om^2)(1-\om^2)= -1 + 2(1-\om)(a+(a-b)\om)$ 
with $(a-b)$ even.

Supposing $p\equiv 1\bmod 12$, then $p=N_2(\alpha)$ for some $\alpha 
=(a+b(1-\om))+i(c+d(1-\om))$.
We can't have $3\mid a$ and $3\mid c$, else $3\mid N(\alpha)$ and we can assume 
that $3\nmid ac$ else we replace $\alpha$ by 
$$ (1+i\om)\alpha = (a-c-3d) +(1-\om)(b+c+2d) +i( a+c+3b+(1-\om)(d-a-2b)). $$
Replacing $\alpha$ by $-\alpha$ or $\pm i \alpha$ we can assume that $a,c\equiv 
2$ mod  3 and writing $3=(1-\om)(2+\om)$ we can write
$$\alpha = -1+(1-\om)(A+B\om) + i\left(2 + (1-\om)(C+D\om)\right). $$
It remains to show that we can take $A,B,D$ all even. We work first with $A,B$. 
We have
\begin{align*}
\om ( -1+(A+B\om)(1-\om)) & = -1 + (1-B+(A-B)\om)(1-\om),\\
\om^2 ( -1+(A+B\om)(1-\om)) & = -1 + (B-A+1+(1-A)\om)(1-\om),
\end{align*}
so if $A,B$ are both odd we can replace  $\alpha$ by $\om \alpha $  and if $A$ 
is odd and $B$ even by $\om^2\alpha$. If $A$ is even and $B$ odd we write
$$ -1 + (A+B\om)(1-\om)= 2+(1-\om)(A-2+(B-1)\om).$$
Notice this does not occur if $C,D$ are both even else $2\mid \alpha$, so 
replacing $\alpha$ by a conjugate of  $-i\alpha$ we reverse the roles of the 
$A,B$ and $C,D$ and then work to make the new $A,B$ both even noting that the 
process keeps $C,D$ even
\begin{align*}
\om ( 2+(C+D\om)(1-\om)) & = 2 + (-D-2+ (C-D)\om)(1-\om),\\
\om^2 ( 2+(C+D\om)(1-\om)) & = 2 + (D-C-2+(-C-2)\om)(1-\om). 
\end{align*}
So suppose that $A,B$ are even. We can't have $C$ even, $D$ odd else $\alpha 
\equiv -1 +i \text{ mod } 2$
and $2\mid N(\alpha)$. If $C$ and $D$ are both odd then we can replace $\alpha$ 
by its conjugate
\begin{align*}  &  -1+ 2(A+B\om^2)(1-\om^2) + i 
\left(2+(C+D\om^2)(1-\om^2)\right) \\
 & \hspace{5ex} =  -1+ 2(A+(A-B)\om)(1-\om) + i\left( 
2+(C+(C-D)\om)(1-\om)\right),    \end{align*}
so we can assume that $D$ is also even.  Observing that
\begin{align*}
\alpha &=  -1+2(A+B\om)(1-\om)+i(2+(C+2D\om)(1-\om) ),\\
\alpha' &  = -1+2(A+B\om^2)(1-\om^2)+i(2+(C+2D\om^2)(1-\om^2), 
\end{align*}
have $\alpha \alpha'=-3\delta + i\rho,$ with
\begin{align*}
\delta & =1+C^2+2A+2C-2CD+4D^2-4(A^2-AB+B^2),\\
\rho  & =  -4-3C+6(2A+2AC+4BD-BC-2AD),
\end{align*}
plainly $C$ even leads to $2\mid \rho$ and an $\P_1$  representation of $p$ and 
if $C$ is odd $2\mid \delta$ and $p$ must be in $\P_2$. 
With 
$$\alpha_1=-1+(2A+2B\om)(1-\om)+i\left( 2 + (C+2D\om)(1-\om)\right)$$ 
we take  $\alpha_2=(\om^2+i)\alpha_1=\delta_2+i\rho_2$ where
\begin{align*}
\delta_2 & = \left( (2B-2A-C-1) -(2A+2D)\om\right)(1-\om),\\
\rho_2 & = 1 +\left((2A+2D-C-2)+(2B-C-2)\om\right)(1-\om), 
\end{align*}
and the second form \eqref{2ndform}  is plain.
Observe that
\begin{align*}
\alpha &=  -1+2(A+B\om)(1-\om)+i(2+(C+2D\om)(1-\om) ),\\
\alpha'' & =  -1+2(A+B\om)(1-\om)-i(2+(C+2D\om)(1-\om)), 
\end{align*}
has
$$ \alpha \alpha'' = 5+ C^2(1-\om)^2 +4(1-\om)h(\om). $$
When   $p$ is in $\P_1$ and $C$ is even
$$ \alpha\alpha''=-1+2(1-\om)\Big(2+\om + 2(C/2)^2(1-\om)+2h(\om)\Big),$$
giving \eqref{norm1mod12} with $B$ odd, while  when $p$ is in $\P_2$ and $C$ is 
odd
$$ \om \alpha \alpha'' 
=-1+4(1-\om)\left(\frac{1}{4}(C^2-1)+\frac{1}{2}(C^2+1)\om +\om h(\om)\right), 
$$
giving \eqref{norm1mod12} with $B$ even.
\end{proof}

We write
$$ h(x) := \frac{x^6-1}{x-1} =\sum_{j=0}^5 x^j,\;\;\; h(1)=6,\; h(-1)=h(\pm 
\om)=h(\pm w^2)=0. $$

\begin{proof}[Proof of Theorem \ref{Q12}] 
As before, we let $m_t$ (e.g., $m_2,m_6$) denote an arbitrary integer 
coprime to $t$.
From the formula  \eqref{E-Q}  we know that $\S(Q_{12})$ consists of measures $M:= M(F):=M_{Q_{12}}(F)$ of the form
\be \label{E-MQ12} M  =ab (c d)^2, \ee
where, if  $F(x,y)=f(x)+yg(x)$,
\[
a:  =f(1)^2-g(1)^2, \; b:=f(-1)^2+g(-1)^2,\]
\[
c:  =|f(\omega)|^2-|g(\omega)|^2,\;
 d:=|f(-\omega)|^2+|g(-\omega)|^2.
 \]
The proof proceeds by a series of steps, followed by their proofs:

\vspace{1ex}

{\bf  \noindent Step 1.} We have $c\equiv a \bmod 3$,  $d\equiv b \bmod 3$ and
\be \label{27div} 3\nmid M   \;\;\; \hbox{ or }\;\;\; 3^3\mid M. \ee
\vspace{1ex}

Since $\om \equiv 1$ mod $(1-\om)$ in $\mathbb Z[\om]$ and $a$ and $c$ are 
integers we have $c\equiv a \bmod 3$. Likewise $d\equiv b \bmod 3$. So  
$cd\equiv ab$ mod 3 and  $M\equiv (ab)^3$ mod $3$. Hence   $3\mid 
M$ iff $3\mid ab$ and $3\mid cd$, proving Step 1.

\vspace{1ex}

{\bf  \noindent Step 2.} All odd integers of the forms $M=m_6$ and $M=3^3m_2$ belong to $\S(Q_{12})$.

\vspace{1ex}
From Step 1 we know that all odd $M$ are of one of these two forms. Conversely,
we obtain all odd integers that are $1 \bmod 4$ satisfying \eqref{27div} as follows:
\begin{align} & \label{1mod12}   M\left(1+t \:h(x) +yt\: h(x)\right) = 
1+12t, \\
 & M\left(1+x(x^3+1)+t\: h(x) +y\left((1+x^3)+ t\:h(x)\right)\right)  = 
5+12t,\label{5plus12t} \\ 
  &   M\left(1+t \; h(x) +y\left((1+x^3)+t\:h(x)\right)\right) =-3^3 
(1+4t).\label{27}
\end{align}
Swapping $f$ and $g$ in \eqref{E-Q} changes the sign, by \eqref{E-ggff}, so that these identities give us all odd integers 
satisfying \eqref{27div}. 

This deals with the odd measures, so we can assume that $M$ is even from now on. In particular, at least one of $ab$ and $cd$ is even.

\vspace{1ex}

{\bf  \noindent Step 3.}  We have $c\equiv d \bmod 2$, $f(1)\equiv f(-1)\bmod 2$, $g(1)\equiv g(-1)\bmod 2$, and $ab$ even iff $f(1)\equiv g(1)\bmod 2$.

\vspace{1ex}

{\bf  \noindent Step 4.} We have $2^4\mid M$, and $2^5\parallel M$ possible only when $c$ and $d$ are both odd, and $f(-1),g(-1)$ are both odd, or both even with $f(-1)/2$, $g(-1)/2$  both odd.

\vspace{1ex}

If $cd$ is even, then from $d\equiv c \bmod 2$ we have that $2^4\parallel (cd)^2$ or $2^6\mid (cd)^2$.
 If $f(1)$ 
and $g(1)$ are both odd then $2^3\mid a$ and $2\parallel b$.  
If both are even then,  writing 
\[
f(1)=2\alpha_1,\quad f(-1)=2\alpha_2, \quad g(1)=2\beta_1, \quad g(-1)=2\beta_2,
\]
we get
$a=4(\alpha_1^2-\beta_1^2)$ and $b=4(\alpha_2^2+\beta_2^2)$.
Hence  $2^4 \mid ab$, and  we get  $2^4\mid M$, 
with $2^5\parallel M$ only as specified.

\vspace{1ex}

{\bf  \noindent Step 5.} All $M$ with $2^4\parallel M$ with 
property \eqref{27div} are possible.

\vspace{1ex}

This is seen  using
\begin{align*}
 & M\left(1+ x^2+t \:h(x) +yt\:h(x)\right) =2^4 (1+6t), \\
 & M\left((1+x^2)+t \:h(x) +y\left((1+x^2)(1+x^3)+t\:h(x)\right)\right) =-3^3 
2^4(1+2t),
\end{align*}
recalling that we can use \eqref{E-ggff} to change the sign on the first of these.

\vspace{1ex}

{\bf  \noindent Step 6.} All $M$ with $2^6\mid M$ with 
property \eqref{27div} are possible.

\vspace{1ex}

This is seen using \eqref{E-ggff} and 
\begin{align*}
& M\left((1+x+x^2+x^4)+t \:h(x) +yt\;h(x)\right) =2^6(1+3t),\\
 & M\left((1+x^2)+t \:h(x) +y\left(-(1+x^3)+t\:h(x)\right)\right) =2^63^3t.
\end{align*}
So the even measures with 
$2^4\parallel M$ or $2^6\mid M$  are as claimed. 

For the rest of the proof we can assume that $2^5\parallel M$ and that $3^\beta\parallel M$. We know that $\beta=0$ or $\beta\ge 3$.

\vspace{1ex}

{\bf  \noindent Step 7.} If $\beta =4$ or $\beta \ge 6$ then $M=2^53^\beta m_6$.

\vspace{1ex}

We get  $2^53^4$ and $2^53^6$ from
 $$M\left( (1-x-x^5)+h(x) 
+y\left((x^4+x^2+1)+(x^2-1)(x+1)^{\delta}\right)\right) =2^5 3^{4+2\delta}. $$
Hence, multiplying by an odd measure, we can  
obtain  any $2^53^\beta m_6$ with  $\beta =4$ or $\beta \geq 6$. 

This just leaves  $\beta =0,3$ or $5$, which we can now assume holds.

\vspace{1ex}

{\bf  \noindent Step 8.} We have $3\nmid b$.

\vspace{1ex}

 For if $3\mid b=f(-1)^2+g(-1)^2$ then $3\mid f(-1),g(-1)$
with $(f(-1)/3)^2+(g(-1)/3)^2$  coprime to 3 or a multiple of $3^2$, 
giving $3^2\parallel b$ or $3^4\mid b$. Then also, from Step 1, 
$d\equiv b\equiv 0\bmod 3$,
 with $d$ contributing an even power of $3$ to $bd^2$, so that $3^4\parallel  bd^2$ or $3^6\mid bd^2$. Also, from $a\equiv c\bmod 3$ we have either 
$3\nmid a$ or $3^3\mid ac^2$.
 Hence  $3^4\parallel M$ or $3^6\mid M$, contrary to assumption. 

\vspace{1ex}

{\bf  \noindent Step 9.} If $3\mid f(-1)g(-1)$ then $p\mid M$ for some prime $p\equiv 5\bmod 12$.

\vspace{1ex}

Supposing first that $3\mid f(-1)g(-1)$, then we have a factor of $M$ of the form 
$$f(-1)^2+g(-1)^2 \; \text{ or}\; (f(-1)/2)^2+(g(-1)/2)^2=A^2+(3B)^2=2k $$ with 
$k$ odd and $2k\equiv 1$ mod 3. Thus $k\equiv 2$ mod 3 is an odd sum of two 
squares, so must contain an odd power of a  prime $p\equiv 2$ mod 3 which  must 
be  a sum of two squares, and hence $p\equiv 5\bmod 12$. 
Conversely suppose that we have a prime $p\equiv 5\bmod 12$ then $2p$ is a sum of 
two squares $2p=a^2+b^2$ where $a$ and $b$ must be odd and one of them a 
multiple of 3, $2p=(1+6A)^2+(3+6B)^2$.

\vspace{1ex}

{\bf  \noindent Step 10.} Conversely every  $2^5p$ with $p\equiv 5\bmod 12$  is an $M$.

\vspace{1ex}
This is because if $2p=(1+6A)^2+(3+6B)^2$ then
$$ M\left( -1-Ah(-x) +h(x) + y\left( x^4+x^2+1+ Bh(-x)\right) \right)   = 
2^5p. $$

\vspace{1ex}

{\bf  \noindent Step 11.} If $3\nmid f(-1)g(-1)$ then $p^2\mid M$ for some prime $p\equiv 5\bmod 6
$.

\vspace{1ex}

For if $3\nmid f(-1)g(-1)$ then 
$$d=|f(-\omega)|^2+|g(-\omega)|^2\equiv f(-1)^2+g(-1)^2 \equiv 2 \text{ mod  } 3$$
 is odd. Hence this term
must be divisible by a prime $p\equiv 5$ mod 6 and $M$ by $p^2$. 

\vspace{1ex}

 {\bf  \noindent Step 12.} Conversely every such $2^5p^2$ with $p\equiv 5\bmod 6$  is an $M$.

\vspace{1ex}

We can write a $p\equiv 5$ mod 6 as
a sum of two norms of elements in $\mathbb Z[w]$, 
$p=N_1(\alpha)+N_1(\gamma)$; we can do this since any integer not of the form 
$9^k(9n+6)$ can be represented by $x^2+xy+y^2+z^2$ -- see Dickson \cite{Dickson}
(in fact  $x^2+xy+y^2+z^2+zw+w^2$ should represents all integers  
\cite{290}). 
Moreover since $p$ is odd  one of the norms must be even and hence an 
element in $2\mathbb Z[w]$ and we can write $p=N_1(\alpha)+4N_1(\beta)$ with 
$N_1(\alpha),N_1(\beta)\equiv 1$ mod 3. By Lemma \ref{normform1}  we can assume 
that $\alpha = -1+ 2(1-\omega)(A+B\omega)$ and $\beta =-1+(C+D\omega)(1-\omega)$. We take
\begin{align*} f(x)&= -1 -(A-Bx)(x^3-1)(1+x)+ h(x) \\
g(x) & = (x^2+x+1)- (C-1-Dx)(x^3-1)(1+x). \end{align*}
Then $f(1)^2-g(1)^2=2^4$, $f(-1)^2+g(-1)^2=2,$ $|f(\omega)|^2-|g(\omega)|^2=1$, 
$|f(-\omega)|^2+|g(-\omega)|^2=N_1(\alpha)+N_1(2\beta)=p$ and $M_G(f(x)+yg(x))=2^5p^2$.

\comments{\begin{align*} f(x)&= -1 -(A-Bx)(x^3-1)(1+x)+ (2m+1)h(x) \\
g(x) & = (x^2+x+1)- (C-1-Dx)(x^3-1)(1+x)+ 2m\: h(x). \end{align*}
Then $f(1)^2-g(1)^2=2^4(1+3m)$, $f(-1)^2+g(-1)^2=2,$ $|f(\omega)|^2-|g(\omega)|^2=1$, 
$|f(-\omega)|^2+|g(-\omega)|^2=N_1(\alpha)+N_1(2\beta)=p$ and $M_G(f(x)+yg(x))=2^5mp^2$.
}

\vspace{1ex}

{\bf  \noindent Step 13.} Any $2^53^\beta km_6$, with $k=p$, $p\equiv 5 \bmod 12$,  or $k=p^2$, 
$p\equiv 5 \bmod 6$ and $\beta =0$ or $\beta \geq 3$ is an $M$.

\vspace{1ex}

Hence we can get any $2^5k$, where  $k=p$, $p\equiv 5 \bmod 12$,  or $k=p^2$, 
$p\equiv 5$ mod 6, and by  multiplicativity any $2^53^\beta km_6$, with  
$\beta =0$ or $\beta \geq 3$.
\end{proof}

\

We remark that the evaluation of $\S(D_{12})$, already known from \cite{dihedral}, can be modelled on the first six steps of the above proof for $Q_{12}$. In place of \eqref{E-MQ12}, for $G=D_{12}$  formula \eqref{E-D} gives $M = ab_1 (c d_1)^2$, with $b_1:  =f(-1)^2-g(-1)^2$ and $d_1:= |f(-\omega)|^2-|g(-\omega)|^2$. Step 1 of the proof also holds for $G=D_{12}$, because $d_1\equiv b_1\bmod 3$. Step 2 also follows using the identities \eqref{1mod12}, \eqref{5plus12t} and \eqref{27}. For $D_{12}$ swapping $f$ and $g$ is not necessary because odd $M\equiv (acd_1)^2\equiv 1\bmod 4$. In Step 3 we still have $c\equiv d_1\bmod 2$. In Step 4 we  also  have $2^3\mid b_1$ when $f(1)$ and $g(1)$ are both odd and $b_1=4(\alpha_2^2-\beta_2^2)$, when both are even, giving $2^4\parallel ab_1$ or $2^6\mid ab_1$, and so $2^4\parallel M$ or $2^6\mid M$. For Steps 5 and 6, we use the same four identities, but, without \eqref{E-ggff}
to change  signs, we need the two extra identities
\begin{align*}
& M\left((1+x^2)(1+x(x^3+1))+t \:h(x) +y\left((1+x^3)(1+x^2)+ t\:h(x)\right)\right)  = 2^4(5+6t),\\
& M\left((1-x)+t \: h(x) +y\left((1+x^2)(1+x^3)+t\:h(x)\right)\right)
=-2^6(1+3t),
\end{align*}
to complete the proof.

\begin{proof}[Proof of Theorem \ref{Z2Z6}] Suppose that $G=\mathbb Z_6\times 
\mathbb Z_2$. From \eqref{E-Z62} we have 
\[
M_G(F)  =ab_1 e_1 e_2 e_3e_4, 
\]
for the integers
\begin{align*}
 & a:  =f(1)^2-g(1)^2,\;\;\;  b_1:  =f(-1)^2-g(-1)^2,\\
 & e_1:=  |f(\omega)+g(\omega)|^2,     \;\; e_2:=|f(\omega)-g(\omega)|^2 ,\;\;\; 
e_1e_2=N_1(f(\om)^2-g(\om)^2),\\
 &  e_3:=  |f(-\omega)+g(-\omega)|^2,     \;\;  e_4:=|f(-\omega)-g(-\omega)|^2,\;\; \; 
e_3e_4=N_1(f(-\om)^2-g(-\om)^2).
\end{align*}

Since $f(\om)\equiv f(\om^2)\equiv f(1)$ mod $(1-\om)$ etc.  in $\mathbb 
Z[\om]$, we readily see that $e_1e_2\equiv a^2$ mod 3
and $e_3e_4\equiv b_1^2$ mod $3$ in $\mathbb Z$. Hence we cannot have 
$3\parallel  ab_1 e_1e_2e_3e_4$  and
\be \label{3div}  3^{\alpha} \parallel M_G(F) \Rightarrow  \alpha=0 \hbox{ or } 
\alpha \geq 2. \ee
Moreover  $3^2\parallel M_G(F)$ implies that $3\parallel ab_1$ and $3\parallel 
e_1e_2e_3e_4$.

Since $f(-x)=f(x)+2u_1(x)$ we readily see that all the $e_i\equiv e_1$ mod $2$. 
Moreover considering factorisation
in $\mathbb Z[\om]$, or by checking that $2\mid N_1(A+B\om)=A^2-AB+B^2$ only 
when  $2\mid A,B$, we see that 
a norm of an element in $\mathbb Z[\om]$ is divisible by an even power of 2. 
Hence
$$ 2^{\beta}\parallel e_1e_2e_3e_4 \Rightarrow \beta =2t, \; t=0 \hbox{ or } 
t\geq 4. $$
From the proof of Theorem \ref{Q12}, and  the comments on $D_{12}$ after it,  we have that
$$ 2^{\beta} \parallel ab_1 \Rightarrow \beta=0,4 \hbox{ or } \beta \geq 6. $$
Hence we certainly have
\be \label{2div}  2^{\beta} \parallel M_G(F) \Rightarrow \beta =0,4 \hbox{ or } 
\beta \geq 6. \ee
Similarly, since $f(-x)^2=f(x)^2+4u_2(x)$ we get $b_1\equiv a$ mod $4$ and 
$e_3e_4\equiv e_1e_2$ mod 4
and if $ab_1$ is odd then $ab_1\equiv a^2\equiv 1$ mod $4$ and if  
$e_1e_2e_3e_4$ is  odd then it is 1 mod 4.
In particular
\be \label{oddity} M_G(F) \text{ odd } \Rightarrow M_G(F)\equiv 1 \text{ mod } 
4. \ee

\noindent
{\bf Case (a)} We begin with the multiples of 27. We can achieve all odd 
multiples  satisfying \eqref{oddity} using \eqref{27}, and all even multiples 
satisfying \eqref{2div} with
\begin{align*}
& M_G\left( (x^2+1)(x^3+1)+ kh(x) + y \left( (x^2+1)+k\:h(x)\right) \right)=-2^4 
3^3(1+2k), \\
& M_G\left( 1+ m\:h(x) + y \left( (1-x-x^5) +m\:h(x)\right) \right)=-2^6 3^3m.
\end{align*}

\noindent
{\bf Case (b)} Suppose next that $3^2\parallel  M_G(F)$.
Consider first the case $e:=e_1e_2e_3e_4$ odd. From the above discussion we know 
that  $e\equiv 1$ mod 4 and $3\parallel e$. So  $e/3\equiv 3$ mod $4$,
and must be divisible by an odd power of a prime $p\equiv 3$ mod 4 with $p\neq 
3$.
Since $e$ is a norm in $\mathbb Z[\om]$ we must have $p\equiv 7\bmod 12$. 
Conversely suppose that $p\equiv 7\bmod 12$.
By Lemma \ref{normform} we have $p=N_1(\alpha_1)$ with $\alpha_1= 1-2\om 
+4(A+B\om)(1-\om)$. Take
\begin{align*} 
f(x) & = (x+1) -(A-Bx)(x^3-1)(x+1) + m\:h(x),\\
g(x) & = x -(A-Bx)(x^3-1)(x+1) + m\: h(x).
\end{align*}
Then $a=3(1+4m)$, $b_1=-1$, 
$e_1e_2=|f(\om)^2-g(\om)^2|^2=|\om-\om^2|^2=3,$
$e_3=N_1( \alpha_1)=p,$ and $ e_4=1, $
giving   $M_G(F)=-9(1+4m)p$, and hence all the odd multiples of $9p$ satisfying 
\eqref{oddity}.
Likewise
\begin{align*} 
f(x) & = (x+1) -(A-Bx)(x^3-1)(x+1) + (m-1)\:h(x),\\
g(x) & = (x^2 +1) - (A-Bx)(x^3-1)(x+1) - m\: h(x),
\end{align*}
has  $a=12(1-2m)$, $b_1=-4$, 
$e_1e_2=N_1(\om-\om^2)=3,$
$e_3= N_1(\alpha_1)$ and $e_4=1$,
and $M_G(F)=9\cdot 2^4(2m-1)p$,
 while
\begin{align*} 
f(x) & = (x+1) +(x^4+x^2+1)-(A-Bx)(x^3-1)(x+1) + (m-1)\:h(x),\\
g(x) & = x -(A-Bx)(x^3-1)(x+1) + m\:h(x).
\end{align*}
has  $a=-24m$, $b_1=8$, 
$e_1e_2=3,$ $e_3e_4=p,$
and $M_G(F)=-3^2\cdot 2^6mp$. Hence we can  achieve all the even multiples of 
$9p$ satisfying \eqref{2div}.

Suppose now that $3^2\parallel M_G(F)$ with $e=e_1e_2e_3e_4$ even, then 
$2^{8+2t}\parallel e$ and $3\parallel e$.
We can suppose that $M_G(F)$ is not divisible by a prime $p\equiv 7\bmod 12$, 
since we already have all such  multiples of these. Hence the odd powers of 
primes, other than $3$, dividing 
$e$ are all 1 mod 4 and hence $e 2^{-8-2t}3^{-1}$ is 1 mod 4. Also the odd $ab_1$  are 
1 mod 4 and the evens are odd multiples of $2^4$ or multiples of $2^6$. 
Thus  either $M_G(F)=2^8m$ or $2^{10}m$ with $m\equiv -1$ mod 4, or  
$2^{12}\parallel M_G(F)$ or $2^{14}\mid M_G(F)$. We can achieve all these:
\begin{align*} 
f(x) & = 2(x^2+1)-(x^4+x^2+1) + m\:h(x),\\
g(x) & = (x^3+1)+ m\; h(x).
\end{align*}
has  $a=-3(1+4m)$, $b_1=1$, 
$e_1e_2=N_1(4\om^2-4)=2^4\cdot 3,$ $e_3e_4=N_1(4\om^2)=2^4$
and  $M_G(F)=-3^2\cdot 2^8(1+4m)$, while
\begin{align*} 
f(x) & = (1-x^2)-2(x^3+1)+(m+1)\:h(x),\\
g(x) & = -x^4+x^2(x^3+1)+ m\: h(x),
\end{align*}
has  $a=3(1+4m)$, $b_1=-1$, 
$e_1=N_1(2(\om^2-1))=2^2\cdot 3,$ $e_2=N_1( 4\om)=2^4,$
$e_3=N_1(2)=2^2$, $e_4=N_1(-2\om^2)=2^2,$
and  $M_G(F)=-3^2\cdot 2^{10}(1+4m)$.

Taking
\begin{align*} 
f(x) & = 2(x^2+1)-h(-x)+m\: h(x),\\
g(x) & = (x^3+1)+ m\:h(x),
\end{align*}
has  $a=12(1+2m)$, $b_1=4$, 
$e_1e_2=N_1(4\om^2-4)= 2^4\cdot 3,$ $e_3e_4=N_1(4\om^2)=2^4,$
and  $M_G(F)=3^2\cdot 2^{12}(1+2m)$. Similarly
\begin{align*} 
f(x) & = 2(x^2+1)-(x^4+x^2+1)+m\;h(x),\\
g(x) & = (x^3+1)-(x^4+x^2+1)+ m\:h(x),
\end{align*}
has  $a=24m$, $b_1=-8$, 
$e_1e_2= 2^4\cdot 3,$ $e_3e_4=2^4,$
and  $M_G(F)=-3^2\cdot 2^{14}m$.

\vspace{1ex}
\noindent
{\bf Case (c)} Finally we deal with the measures coprime to 3.  We achieve any 
value $1\bmod 12$ with \eqref{1mod12}.
We can also get all odd multiples of $2^4$ or multiples of $2^6$  when  the 
multiple is 1 mod 3:
\begin{align*}
& M_G\left( (x^2+1)+ m\: h(x) + y m\: h(x) \right)=2^4(6m+1), \\
& M_G\left( (x^4+x^2+1)+ m\:h(x)+ y \left( 1+m\:h(x)\right) \right)=2^6 (3m+1).
\end{align*}
Since odd measures are 1 mod 4 and even measures are divisible by  exactly $2^4$ 
or at least $2^6$  this leaves the measures 5 mod 12 or $-2^4(6m+1)$ or 
$-2^6(3m+1)$. 

Suppose that $p\equiv 1\bmod 12$ is in $\P_1$  then, by Lemma \ref{normform}, 
$p=N_1(\alpha)$ for some
\be \label{formk}  \alpha =-1-2\om(1-\om) +4(A+B\om)(1-\om). \ee
Suppose $p=12t+5$ then $p=-1-2\om(1-\om) + 4((2t+1)+(t+1)\om)(1-\om)$ and 
$p^2=N_1(\alpha)$ for an $\alpha$ of the form
\eqref{formk} and if $p_1,p_2\equiv 7\bmod 12$ then by Lemma \ref{normform} 
$p_1p_2=N_1(\alpha)$ for some
\begin{align*} \alpha  & =\big( -1 + 2(1-\om)+4(1-\om)K_1(\om)\big)\big( 
1+2(1-\om)(1+\om)+4(1-\om)K_2(\om)\big) \\
 & = -1-2\om(1-\om) +4(1-\om)K_3(\om),
\end{align*}
which is also of the form \eqref{formk}. Hence for $k=p$ with $p$ in $\P_1$ or $k=p^2$ 
with $p\equiv 5\bmod 12$ or $k=p_1p_2$ 
with $p_1,p_2\equiv 7\bmod 12$ we can write $k=N_1(\alpha)$ for an $\alpha$ of 
the form \eqref{formk} and taking
\begin{align*}
f(x) &= x(x^2+x+1) - (A-Bx)(x^3-1)(1+x)+  m\:h(x)),\\
g(x) & = x(x+1)  - (A-Bx)(x^3-1)(1+x)+m\:h(x), 
\end{align*}
or
\begin{align*}
f(x) &= x^2(1-x^3)- (A-Bx)(x^3-1)(1+x)+  m\:h(x),\\
g(x) & = x(x+1)  - (A-Bx)(x^3-1)(1+x)+m\:h(x), 
\end{align*}
or
\begin{align*}
f(x) &= x(x^2+x+1)- (A-Bx)(x^3-1)(1+x)+  m\:h(x),\\
g(x) & = x(x+1) -(x^4+x^2+1) - (A-Bx)(x^3-1)(1+x)-m\:h(x), 
\end{align*}
we have $e_1e_2= N_1(-1)=1$, $e_3=N_1(\alpha)=k$, $e_4=N_1(-1)=1$
with, 
$$ (a,b_1)=(5+12m, 1),\;\;\; (-2^2(6m+1), 2^2),\;\; \text{ and } \;\; 
(2^3(3m+1),-2^3),$$
respectively, and $M_G(F)=(5+12m)k,$ $-2^4(6m+1)k$ and $-2^6(3m+1)k$. 

So we can achieve the \eqref{type2} and need just consider cases that do not 
contain the square of a prime 5 mod 12 or two primes 7 mod 12 or a prime in 
$\P_1$. Suppose first that  $e=e_1e_2e_3e_4$ is odd, then since it is 1 mod 4,  
it must contain only primes from $\P_2$ and squares of primes $11\bmod 12$. Since 
$e\equiv 1$ mod 3, to get a measure $2$ mod 3 we must have one of  
$a=f(1)^2-g(1)^2,$ $b_1= f(-1)^2-g(-1)^2\equiv 1$ mod 3 and the other $-1$ mod 
3.
Replacing $\pm x$, $\pm f$, $\pm g$ we can assume that 
$f(1)\equiv 1$ mod 3, $g(1)\equiv 0$ mod 3, $f(-1)\equiv 0$ mod 3, $g(-1)\equiv 
1$ mod 3. Since $f(\pm \om)\equiv f(\pm 1)$ mod $(1-\om)$, $g(\pm\om)\equiv 
g(\pm1)$ mod $(1-\om)$, we can write
\begin{align*}  f(\om)+g(\om) & =1+(1-\om)(A_1+B_1\om),\\
 f(\om)-g(\om) & =1+(1-\om)(A_2+B_2\om),\\
 f(-\om)+g(-\om) & =1+(1-\om)(A_3+B_3\om),\\
 f(-\om)-g(-\om) & =-1+(1-\om)(A_4+B_4\om),
\end{align*}
and since $f(-\om)\equiv f(\om)$ mod $2$ and $g(-\om)\equiv g(\om)$ mod 2 and 2 
is prime we readily see that the $A_i$ have the same parity and the $B_i$ all 
have the same parity. 
Multiplying by $x$ or $x^2$ as necessary we can assume that $A_1$ and $B_1$ are 
both even and hence all the $A_i$ and $B_i$ are all even. All the primes in 
$\P_2$ or $11\bmod 12$ factor in $\mathbb Z[\om]$ into $1+4(A+B\om)(1-\om)$ times 
a unit, and (however we factor) an expressions of this type  with $A_i$, $B_i$ 
even will have $A_i$, $B_i$  both a multiple of 4. 
Hence $g(\om)\equiv 0$ mod 2 and $g(-\om)\equiv 1$ mod 2 contradicting 
$g(\om)\equiv g(-\om)$ mod 2, so we have no extra measures.
This leaves $e$ even and hence $2^{8+2\ell_1}\parallel e$ and 
$2^{\ell_2}\parallel ab_1$
with $\ell_1=0,4$ or $\ell_1\geq 6$ and hence
$2^t \parallel M_G(F)$ with $t=8,10,12$ or $t\geq 14$.
We can get all the multiples of $2^{12}$ and $2^{14}$:
\begin{align*}
f(x) & = (x^2+x+1) + m\:h(x),\\
g(x) & = (x^2-x+1) + m\: h(x),
\end{align*}
and
\begin{align*}
f(x) & =( x^2+x+1)-(x^4+x^2+1) + m\: h(x),\\
g(x) & = (x^2-x+1)-(x^4+x^2+1) - m\: h(x),
\end{align*}
have $e_1e_2=N_1(-4\om^2)=2^4$, $e_3e_4=N_1(4\om^2)=2^4$ with, respectively, 
$$a=2^3(1+3m),\;\; b_1=-2^3 \;\;  \text{  and }  \;\; a=-2^2(1+6m), \; b_1=2^2, 
$$
giving $M_G(F)=-2^{14}(3m+1)$ and $-2^{12}(6m+1)$.

For $t=8$, 10 we have  $ab_1\equiv 1$ mod 4. If $e$ does not contain a prime  7 
mod 12 then
 $e=2^{\ell_2}n'$ with $\ell_2=8$ or $10$ and  $n'\equiv 1$ mod 4 and 
$M_G(F)=2^8n$ or $2^{10}n$ with $n\equiv 1$ nod 4, and all these 2 mod 3 are 
obtainable:
\begin{align*}
f(x) & =( x^2+x+1)+ m\:h(x),\\
g(x) & = (x^3+1)+  m\:h(x),
\end{align*}
has $e_1e_2=N_1(-2^2)=2^4$, $e_3=e_4=N_1(-2\om)=2^2,$ 
$a=(5+12m),$ $b_1=1$,  
and
\begin{align*}
f(x) & =-1-2x^3- m\:h(x),\\
g(x) & = -1+(x^2+x+1)+m\:h(x),
\end{align*}
has $e_1e_2=N_1(2^3)=2^6$, $e_3=N_1(-2\om)=2^2$, $e_4=N_1(2+2\om)=2^2$ with
$a=(5+12m),$ $b_1= 1, $
giving $M_G(F)=2^8(5+12m)$ and $2^{10}(5+12m)$ respectively.

If $e$ contains a prime $p\equiv 7\bmod 12$ (since it has at most one the 
remaining odd primes are in  $\P_2$ or squares of primes 11 mod 12)  we have 
$M_G(F)=2^4pn$ or $2^6pn$ with $n\equiv 1$ mod 4 and all are again achievable. 
We write
$$p=N_1(1+2(A+B\omega)(1-\omega)), \;\;\;  k(x):=-x(A-Bx)(x^3-1)(1+x),$$ 
then adding $k(x)$ to both $f(x)$ and $g(x)$  in the previous two examples, 
instead of 
$e_3=N_1(-2\om)$   we have  $e_3=N_1(-2\om-4w(A+B\omega)(1-\omega))=2^2p$, with the other 
values remaining  unchanged, and $M_G(F)=2^8(5+12m)p$ and $2^{10}(5+12m)p$.
\end{proof}

\begin{proof}[Proof of Theorem \ref{Z12}] We can write
$$ F(x)=f(x^2)+xg(x^2)\in \mathbb Z[x],\;\;\; f(x)=\sum_{j=0}^5 a_jx^j,\;\; 
g(x)=\sum_{j=0}^5 b_jx^j $$
so that 
$$M_G(F)=\prod_{j=0}^{11} F(\om_{12}^j) =a b  s_1s_2$$ 
where, as before $a=f(1)^2-g(1)^2$, $b=f(-1)^2+g(-1)^2$, and
$$ s_1:=N_1(f(\om)^2-\om g(\om)^2),\;\;s_2:=N_1(f(-\om)^2+\om g(-\om)^2).$$
Observe that $s_1\equiv a^2$ mod 3 and $s_2 \equiv b^2$ mod 3, so either 
$3\nmid M_G(F)$ or $3^2\mid M_G(F)$.

\noindent
{\bf Case (i)} Suppose that $M_G(F)$ is odd. Note that $M_G(x)=-1$ giving us 
both $\pm m$ for measures $m$. 
For the odd values we can get any integer coprime to $6$ using 
$ M_G\left(1+ m\:k(x)\right)=1+12m$ and
\[
M_G\left(\frac{x^5-1}{x-1}+ 
m\:k(x)\right)=5+12m,\;\; k(x):=\frac{x^{12}-1}{x-1}, 
\]
and any odd multiple of $27$ with
$$ M_G\left(1+x^3+x^6+t \:k(x)\right)=3^3(1+4t). $$
 This leaves  $9m,$  $2,3\nmid m.$ 
Moreover  we  cannot have $3\mid b$ unless $3\mid f(-1),g(-1)$ and $3^2\mid 
b$ and $3^3\mid M_G(F)$. So we assume that $3\nmid s_2b$
 and $3\parallel a$, $3\parallel s_1$.

Suppose first that $3\nmid f(-1)g(-1)$. Then $b\equiv 2$ mod 3 and hence is 
divisible by an odd power of some 
prime $p\equiv 2$ mod 3 which must itself be a sum of two squares, and $p\equiv 
5\bmod 12$. Conversely suppose that 
$p\equiv 5\bmod 12$ then $p$ is a sum of two squares both must $\pm 1$ mod 3 with 
one even and one odd. So 
changing signs as necessary $p=(6A+1)^2+(6B+2)^2$. Taking
\begin{align*}
f(x)= &  (x^2+1) -B(x^4+x^2+1)(x-1) + m\:h(x),    \\
g(x)= &  1 -A(x^4+x^2+1)(x-1) + m\:h(x),
\end{align*}
we have $a=3(1+4m)$, $ b= (6B+2)^2+(6A+1)^2=p$, $s_1=N_1(\om^2-\om)=3$, 
$s_2=N_1(\om^2+\om)=1$
and $M_G(F) = 9p(1+4m)$. 

Suppose now that $3\mid f(-1)g(-1)$.
Writing
\begin{align}  \label{aandalpha} u & := |f(\om)|^2+|g(\om)|^2,\;\;\; v:=\om 
f(\om)g(\om^2)+\om^2f(\om^2)g(\om), \nonumber \\
\alpha & :=|f(-\om)|^2-|g(-\om)|^2,\;\;    \beta:=  \om 
f(-\om)g(-\om^2)+\om^2f(-\om^2)g(-\om), \end{align}
then $u,v,\alpha,\beta $ are integers with
\begin{align*} ( f(\om)+\om^2g(\om))(f(\om^2)+\om g(\om^2)) & = u + v, \\
( f(\om)-\om^2g(\om))(f(\om^2)-\om g(\om^2)) & = u -v, \\
 ( f(-\om)+i\om^2g(-\om))(f(-w^2)+i\om g(-\om^2)) & = \alpha + i\beta, \\
( f(-\om)-i\om^2g(-\om))(f(-\om^2)-i\om g(-\om^2)) & = \alpha - i\beta, 
\end{align*}
and  
$$s_1  =(u+v)(u-v)=u^2-v^2,\;\;\;\; 
 s_2  =(\alpha+i\beta )(\alpha -i\beta)=\alpha^2+\beta^2.$$
Writing $f(-\om)=f(\om) +2h_1(\om)$, $f(\om)=f(1)+(1-\om)h_2(\om)$, 
$f(-\om)=f(-1)+(1-\om)h_3(\om)$ etc. and observing that since $3\parallel  a$ we 
must have $3\nmid f(1)g(1)$, we readily get
$$ u\equiv f(1)^2+g(1)^2 \equiv 2 \text{ mod } 3,\;\;\;  \alpha \equiv 
f(-1)^2-g(-1)^2 \equiv \pm 1 \text{ mod } 3, $$
$$ u\equiv \alpha \text{ mod } 2,  \;\;\; v\equiv \beta \text{ mod } 2, $$
$$ v\equiv 2f(1)g(1) \equiv \pm 1 \text{ mod } 3,\;\;\;  \beta \equiv 
2f(-1)g(-1) \equiv 0 \text{ mod } 3. $$
Suppose first that $u$ is odd. Then $v$ is even and $s_1=u^2-v^2\equiv 1$ mod 4. 
Hence, in addition to the single prime 3,  $s_1$ must
contains an odd power of a prime $p\equiv 3$ mod 4. Since it is a norm  in 
$\mathbb Z[\om],$ the prime must be 
$1$ or $7\bmod 12$ and $p\equiv 7\bmod 12$.  Suppose that $p\equiv 7\bmod 12$.  
From Lemma  \ref{normform} we have $p=N_1(\alpha)=N_1(-\alpha \om)$
where  $\alpha= -1+2(1-\om)(2A+1+2B\om)$ and
$$ -\alpha \om = 1-\om^2 +4(\om^2+A(\om^2-\om)+B(1-\om^2))= 1-\om^2 +4(C+D\om)$$
for some integers $C,D$. Taking 
\begin{align*}
f(x) & = x^3+1 +(C+Dx)(1-x)(x^3+1)+ m\:h(x),\\
g(x) & = x +x(C+Dx)(1-x)(x^3+1)+ m\: h(x),\end{align*}
gives $a=3(1+4m)$, $b=1$, $s_2=N_1(1)=1$ and $f(\om)-\om^2g(\om)= 1$ with
$$ f(\om)+\om^2g(\om)= 3+4(1-\om)(C+D\om)=(1-\om)(-\alpha \om), $$
and $s_1=3p$.
Suppose next that $u$ is even, so $\alpha \equiv \pm 1$ mod $3$ is even and 
$\beta$ is an odd multiple of $3$. Since we have produced all multiples of 
primes $5$ or $7\bmod 12$ we suppose that all the primes dividing 
$s_2=\alpha^2+\beta^2$ are 1 or 11 mod 12. Since $\alpha$ and $\beta$ are both 
non-zero we have a non-trivial factorization    in $\mathbb Z[i]$ and $s_2$ must 
contain at least
one prime $p\equiv 1\bmod 12$. Moreover we can't have all these primes in $\P_2$, 
since 
$$ (6k +  (6t+1)i)(6k'\pm (6t'+1)i)= 6k''\mp 1 + 6t''i  $$
all the possible factorisations of such an  $n$  in $\mathbb Z[i]$ would produce 
an $\alpha + i\beta$ of the form $6k\pm (6t+1)i$
not $(6k+3)+(6t+2)i$,  so at least one of the primes is in $\P_1$. Conversely 
suppose that $p\equiv 1\bmod 12$ is in $\P_1$. Then by Lemma \ref{normform} we can 
write $p=N_2(\alpha_2)$ with
$$\alpha_2 =(1+2A+2B\om)(1-\om)+i\left(1+2(C+D\om)(1-\om)\right). $$
We take
\begin{align*}
f(x) & = 1+x -(A-Bx)(1+x)(x^3-1)+ m\; h(x),\\
g(x) & = x -x(C-Dx)(1+x)(x^3-1)+ m\;h(x),\end{align*}
which gives $a=3(1+4m)$, $b=1$, $s_1=N_1(w-1)=3$ and
$$ f(-\omega)= (1-\omega)(1+2A+2B\omega), \;\;\; -w^2g(-\omega)=1+(1-\omega)(2C+2D\omega), $$
and $s_2=N_2(f(-\omega)-iw^2g(-\omega))=p$ and $M_G(F)=9(1+4m)p$.

\vspace{1ex}
\noindent
{\bf Case (ii)} Suppose that $M_G(F)$ is even. Observe, as in the proof of 
Theorem \ref{Q12}, that $ab$ is odd or a multiple of $2^4$ and since 
$s_1\equiv s_2$ mod 2,  with a norm of a $\xi+\eta w$  even only when it is a power 
of 4 we have $s_1s_2$
is odd or an odd multiple of $2^4$ or a multiple of $2^6$. 
We can achieve all odd multiples of $2^4$ and all multiples of $2^6$  that are 
coprime to 3, and multiples of $2^4$ divisible by $3^2$.
\begin{align*}
& M_G\left( 1+x^4 + m\:k(x)\right)= 2^4 (1+6m),\\
& M_G\left( (1+x^4)^2 -\;h(-x^2)+m\:k(x)\right) =2^6 (1+3m),\\
 & M_G\left( x-1 + m\;k(x)\right) = 2^4 3^2 m. \end{align*}

That just leaves $2^5m$ with $(m,6)=1$. 
Suppose first that $3\mid f(-1)g(-1)$. As  in the discussion in Theorem 
\ref{Q12} we must have  $f(-1)^2+g(-1)^2=2n$ or $8n$,  with  here $n\equiv 2$ 
mod $3$ odd. Hence $n$ is divisible by an odd power of an odd  prime $p\equiv 2$ 
mod 3 and since $p$ factors in $\mathbb Z [i]$ must also be 1 mod 4 and $p\equiv 
5$  mod 12. Conversely suppose that $p\equiv 5\bmod 12$. Then $2p$  is a sum of 
two squares
$2p=r^2 + s^2.$  Since it is $1$ mod $3$, replacing $r$ by $-r$  as necessary,  
we can assume  that $3\mid s$ and  $r\equiv 1$ mod 3 and $2p= (1-3A)^2 + (3B)^2. 
$
The choice
\begin{align*}  f(x) & =(x^2+1) +  A(x-1)(x^4+x^2+1) +m\:h(x),\\
 g(x)& = Bx(x-1)(x^4+x^2+1) +m\:h(x), 
\end{align*}
has $a=4(1+6m)$,  $s_1=s_2=N_1^2(1+\om^2)=1$, 
$b= (2-6A)^2+ (6B)^2=8p, $
and $M_G(F)=32(1+6m)p$.

Suppose that $3\nmid f(-1)g(-1)$. Defining $u,v$ and $\alpha,\beta$ as in 
\eqref{aandalpha} and observing that since $3\nmid a$ we must have $3\mid f(1)$ 
or $g(1)$ but not both, we readily get
$$ u\equiv f(1)^2+g(1)^2 \equiv 1 \text{ mod } 3,\;\;\;  \alpha \equiv 
f(-1)^2-g(-1)^2 \equiv 0 \text{ mod } 3, $$
$$ u\equiv \alpha \text{ mod } 2,  \;\;\; v\equiv \beta \text{ mod } 2 , $$
$$ v\equiv 2f(1)g(1) \equiv 0\text{ mod } 3,\;\;\;  \beta \equiv 2f(-1)g(-1) 
\equiv \pm 1 \text{ mod } 3. $$
Suppose first that $u$ is even. Since $2\nmid s_1=u^2-v^2$ we get $v$ is odd and 
$s_1\equiv 3$ mod 4  is a positive integer 
so must be divisible by an odd power of a  prime $p\equiv 3$ mod 4. Since $n$ is 
the norm of an element of $\mathbb Z[\om]$ which is a UFD the prime $p$ must 
split in $\mathbb Z[\om]$ and so $p\equiv 7\bmod 12$. Conversely suppose that 
$p\equiv 7\bmod 12$. Then we can, by Lemma \ref{normform}, write 
$p=N_1(\alpha)$ with 
$\alpha = -1+ 2(2A+1+2B\omega)(1-\omega)$.  We take
\begin{align*}
f(x) &  =(x^2+1) + (A+Bx)(x^3+1)(1-x)+ m\:h(x), \\
g(x)  & =x(1-x) + x(A+Bx)(x^3+1)(1-x)+ m\:h(x),
\end{align*}
then $a=4(1+6m)$, $b=8$,  $s_2=N_1(\om^2+\om)=1$,
with 
$$ f(\om)+\om^2g(\om)=-1+2(2A+1+2B\om)(1-\om),   \;\; f(\om)-\om^2g(\om) =  -1, 
$$
and $s_1=N_1(\alpha)=p$ and $M_G(F)=32(6m+1)p$. 

Suppose that $u$ is odd. Then $\alpha$ is an odd multiple of 3 and $\beta \equiv 
\pm 1$ mod 3  is even.
We have already obtained all multiples of primes $5$ or $7$ mod $12$ so we can 
assume that $s_2=\alpha^2+\beta^2$  contains
only  primes $1$ or $11\bmod 12$, but the primes $11\bmod 12$ do not split in 
$\mathbb Z[i]$ so we can not obtain a factorisation of $s_2$ with $\alpha$ and 
$\beta$ both non-zero from just those primes.  Thus $s_2$ contains at least one prime $p\equiv 1$ mod 
12 with at least one of these in $\P_1$.
Conversely suppose that $p$ is in $\P_1$. By Lemma \ref{normform}  there is an
$$\alpha_1=-1+2(A+B\om)(1-\om) + i(2+2(C+D\om)(1-\om)).$$
with $p=N_2(\alpha_1)$, and
\begin{align*}
f(x) & = x^2(x^2+1)- (A-Bx)(x^3-1)(1+x)+m\:h(x), \\
g(x) & = x(x^3-1)+ (C-Dx)x(x^3-1)(1+x)+m\:h(x),
\end{align*}
has  $a=4(1+6m)$,  $b=8$, $s_1=N_1(1)=1$, with
$$ f(-\om)+ i\om^2g(-\om)=   -1 +2(A+B\om)(  1-\om) + i\left(  2  
+2(C+D\om)(1-\om)    \right) =\alpha_1, $$
so that $s_2=N_2(\alpha_1)=p$ and $M_G(F)=32(6m+1)p.$
\end{proof}

\section{Proof of Theorem \ref{Z3Z3}}

 Note that $M_G(-F)=-M_G(F)$ so we take both signs for each value. In  
\cite{dilum} it was shown that the measures for $\mathbb Z_p \times \mathbb Z_p$ 
coprime to $p$ are exactly the $(p-1)$st  roots of unity mod $p^2$. For $p=3$ 
these are $\pm 1$ mod 9 achieved with
$$ M_G( \pm 1 + m(x^2+x+1)(y^2+y+1))=9m\pm 1. $$
For the multiples of 3 we note, writing $\om = e^{2\pi i/3}$,  that in $\mathbb 
Z [\om]$ 
$$ F(\om^i,\om^j)\equiv F(1,1) \text{ mod } (1-\om),   $$
in particular if $3\mid M_G(F)$  then $3\mid F(1,1)$ and $(1-\om)\mid 
F(\om^i,\om^j)$ in $\mathbb Z[\om]$ for all $i,j$, and $3(1-\omega)^8\mid M_G(F)$ and 
$3^5\mid M_G(F)$. In fact more is true. We write our polynomial in the form 
\begin{align*}  F(x,y) =(A_0+A_1(x&-1) +   A_2(x-1)^2)&\\   &+ (B_0+B_1(x-1)   + 
B_2(x-1)^2)(y-1)&\\
        &\qquad + (C_0+C_1(x-1) + C_2(x-1)^2)(y-1)^2.& 
\end{align*}
If $3\mid F(1,1)$ we  have $3\mid A_0$ and hence
$$ F(\om^i,\om^j)\equiv A_1(\om^i-1) + B_0(\om^j-1) \text{ mod } (1-\om)^2. $$
If $3\mid A_1$ or $3\mid B_0$ we have $(1-\om)^2\mid F(x,y)$ for 
$(x,y)=(1,\om),(1,\om^2)$ or $(\om,1),(\om^2,1)$ and we gain an additional 3. 
Similarly if $A_1\equiv B_0$ mod 3 we have $(1-\om)^2\mid F(x,y)$ for 
$(x,y)=(\om^2,\om),(\om,\om^2)$
and if $A_1\equiv -B_0$ mod 3 we have $(1-\om)^2\mid F(x,y)$ for 
$(x,y)=(\om,\om),(\om^2,\om^2)$. Hence we must have $3^6\mid M_G(F)$. 

Noting that $M_G(-F)=-M_G(F)$  we can obtain all multiples of $3^6$ using
\begin{align*}
& M_G(1+2x+m(x^2+x+1)(y^2+y+1)) = 3^6(1+3m), \\
& M_G(1+2x-x(y^2+y+1)+m(x^2+x+1)(y^2+y+1))  = 3^7m. \qed
\end{align*}

\end{document}